\documentclass[hidelinks,onefignum,onetabnum]{siamart250211}

\usepackage{lipsum}
\usepackage{xurl}
\usepackage{amsfonts}
\usepackage{graphicx}
\usepackage{epstopdf}
\usepackage{amsmath,amstext,amssymb}
\usepackage{mathrsfs}
\usepackage{booktabs}
\usepackage{tabularx}
\usepackage{bm}
\usepackage{multirow}
\usepackage[algo2e, ruled, linesnumbered]{algorithm2e}
\ifpdf
\DeclareGraphicsExtensions{.eps,.pdf,.png,.jpg}
\else
\DeclareGraphicsExtensions{.eps}
\fi

\newtheorem{example}{Example}
\newsiamremark{remark}{Remark}
\newsiamremark{hypothesis}{Hypothesis}
\crefname{hypothesis}{Hypothesis}{Hypotheses}
\newsiamthm{claim}{Claim}
\newsiamremark{fact}{Fact}
\crefname{fact}{Fact}{Facts}

\headers{XI-PINN for moving interface problems}{Ran Bi, Weibing Deng, and Yameng Zhu}

\title{Extended Interface Physics-Informed Neural Networks Method for Moving Interface Problems\thanks{Submitted to the editors DATE.
		\funding{This work was 
			partially supported by the NSF of China grant  12171237, and by the
			Ministry of Science and Technology of China grant 2020YFA0713800.}}}

\author{Ran Bi\thanks{School of Mathematics, Nanjing University, Nanjing 210093, People’s Republic of China
		{(\email{ranbi@smail.nju.edu.cn}, \email{wbdeng@nju.edu.cn}, \email{dg21210022@smail.nju.edu.cn}).}}\and Weibing Deng\footnotemark[2] \and Yameng Zhu\footnotemark[2]
}

\usepackage{amsopn}

\ifpdf
\hypersetup{
  pdftitle={XI-PINN},
  pdfauthor={Ran Bi and Weibing Deng and Weibing Deng}
}
\fi

\definecolor{gold}{rgb}{0,0,1}

\begin{document}

\maketitle

\begin{abstract}
	Physics-informed neural networks (PINNs) have emerged as an effective class of mesh-free methods for solving partial differential equations (PDEs), particularly on complex geometries. In this paper, we introduce an Extended Interface Physics-Informed Neural Network (XI-PINN) framework designed to solve parabolic moving interface problems. The proposed method employs a level set function—which can be either analytically prescribed or learned via a neural network—to capture the moving interface. Furthermore, we establish an a priori error analysis for the XI-PINN method and derive error bounds for the approximation. {Extensive numerical experiments are provided to validate the accuracy and robustness of the framework, and its applicability is further demonstrated by solving the Oseen equations.}
\end{abstract}



\section{Introduction}
Interface problems arise in many scientific and engineering models that involve multiple materials with different chemical or physical properties (see~\cite{greengard1994numerical, sussman1999efficient} and references therein). In these models, the computational domain is composed of subdomains that are separated by smooth curves (or surfaces) known as interfaces. The interface geometry itself may be dynamic; that is, the interface between different subdomains varies in time. The low global regularity of the solutions and the evolving interface pose challenges for numerical simulations, particularly when the interface undergoes large deformations.

Let $\Omega \subset \mathbb{R}^d$, $d\in \mathbb{N}$, be a fixed bounded domain and let $\Omega^+(t)$ and $\Omega^-(t)$ be two subdomains of $\Omega$ separated by the evolving interface $\Gamma(t)$ on a time interval $[0,T_{\text{end}}]$. In this paper, we consider the following parabolic equation:
\begin{equation}\label{eq1.1}
	\left\{
	\begin{aligned}
		\partial_t u - \nabla \cdot \left( \beta \nabla u\right)  &= f \quad  &&\mathrm{in}\; \Omega = \Omega^+(t) \cup \Omega^-(t), \quad  t \in [0, T_{\text{end}}],  \\
		u\left( \cdot, t\right) &= g \quad  &&\mathrm{on}\; \partial \Omega, \quad t \in [0, T_{\text{end}}], \\
		u\left( \cdot, 0\right) &= u_0 \quad &&\mathrm{in}\; \Omega = \Omega^+(0) \cup \Omega^-(0),
	\end{aligned}\right.
\end{equation}
with the jump conditions on the moving interface $\Gamma(t)$:
\begin{align}
	[u]_{\Gamma(t)} &:= u^+|_{\Gamma(t)} - u^-|_{\Gamma(t)} = h_D, \quad t\in [0, T_{\text{end}}]  \label{eq1.4},\\
	[\beta \nabla u \cdot \mathbf{n}]_{\Gamma(t)}&:= \beta^+ \nabla u^+ \cdot \mathbf{n}|_{\Gamma(t)} - \beta^- \nabla u^- \cdot \mathbf{n}|_{\Gamma(t)} = h_N, \quad t\in [0, T_{\text{end}}]  \label{eq1.5},
\end{align}
where $f \in L^2(0,T_{\text{end}}; L^2(\Omega))$, and $\mathbf{n}$ is the unit normal vector to $\Gamma(t)$ pointing from $\Omega^-(t)$ to $\Omega^+(t)$. The restrictions of $u$ on $\Omega^+(t)$ and $\Omega^-(t)$ are denoted by $u^+ = u|_{\Omega^+(t)}$ and $u^- = u|_{\Omega^-(t)}$.
The coefficient function $\beta$ is a piecewise positive constant defined as follows:
\begin{equation}
	\beta(\bm{x}, t) = \begin{cases}
		\beta^+, & \text{for } \bm{x} \in \Omega^+(t), \\
		\beta^-, & \text{for } \bm{x} \in \Omega^-(t). 
	\end{cases}
	\label{eq1.6}
\end{equation}

Furthermore, we assume that there is a prescribed velocity field $\mathcal{V}(\bm{x}, t)$ that governs the movement of the interface, i.e.,
\begin{equation}
	\dfrac{\mathrm{d}\bm{x}}{\mathrm{d}t} = \mathcal{V}(\bm{x}, t), \quad \bm{x} \in \Gamma(t).  \label{eq1.7}
\end{equation}

The parabolic moving interface problem (\ref{eq1.1})--(\ref{eq1.7}) appears in many applications, such as the Stefan problem~\cite{chen1997simple} and the Burton--Cabrera--Frank-type model~\cite{caflisch2003analysis}. For instance, $u$ represents the temperature and $\mathcal{V}$ is determined by the heat flux across the interface in Stefan problems. Under appropriate assumptions on the initial, boundary, and jump conditions (see Section~\ref{sec3} for detailed discussions), it can be proved that Eqs.~(\ref{eq1.1})--(\ref{eq1.7}) admit a unique solution $u$ belonging to the space $L^2(0, T_{\text{end}}; H^1(\Omega))$.

It is well known that moving interface problems pose significant challenges for numerical simulations because the computational domain itself evolves. When employing traditional finite element methods (FEMs), one must ensure the generation of body-fitted meshes that conform to the interface; otherwise, the accuracy of the numerical solution will be severely degraded (cf.~\cite{babuvska2000can}). Chen and Zou~\cite{chen1997simple} studied the linear FEM on nearly fitted quasi-uniform meshes and proved optimal-order error estimates up to some logarithmic factors. In general, mesh regeneration should be minimized or even avoided whenever possible, as mesh generation is notoriously laborious and time-consuming, particularly for complex interface geometries and higher-dimensional PDEs. Consequently, arbitrary Lagrangian--Eulerian methods~\cite{lan2020finite} and space-time finite element methods~\cite{tezduyar1992new} were proposed to overcome this difficulty.

To completely eliminate the need for mesh generation in solving PDEs, unfitted methods have attracted significant attention over recent decades. Unfitted methods allow the interface to cut through mesh elements, while special techniques are required to incorporate the jump conditions across the interface. To handle interface-cutting elements, one class of methods enforces the jump conditions through computational schemes such as CutFEM~\cite{burman2015cutfem, chen2023arbitrarily} and IPFEM~\cite{huang2017unfitted}. Another class of methods constructs specialized basis functions within interface elements that are designed to satisfy the interface jump conditions, such as GFEM~\cite{babuvska1983generalized, zhang2022condensed}, XFEM~\cite{dolbow2001extended}, and IFEM~\cite{li1998immersed, adjerid2024high}. Guo~\cite{guo2021solving} presented the first discrete analysis and optimal error estimates for the backward Euler IFEM applied to parabolic moving interface problems. All the aforementioned methods are effective for solving interface problems and eliminate the need for body-fitted mesh regeneration at each time step. However, they still require either careful handling of small cut elements~\cite{chen2021adaptive} near the interface or the construction of specialized basis functions~\cite{adjerid2024high} within interface elements that satisfy the jump conditions.

Solving interface problems via deep learning methods has attracted extensive attention in recent years. As mesh-free methods, neural networks are particularly well suited for addressing problems in computational domains with complex geometries and higher-dimensional settings. Deep learning methods effectively mitigate the curse of dimensionality and eliminate the need for handling solution meshes. The well-known PINNs~\cite{raissi2019physics} and DeepONet~\cite{lu2021learning}, along with various improved methods~\cite{tseng2023cusp, hu2022discontinuity, wu2022inn, ying2024accurate, wu2024solving, bi2025xi}, have been successfully used to solve elliptic interface problems, yielding promising results. {To address more challenging moving interface problems, Wang et al.~\cite{wang2021deep} proposed a multi-network framework for Stefan problems. In their approach, two distinct networks approximate the temperature fields in the two phases, and an additional network captures the interface evolution, serving essentially to classify the domain points. Lin et al.~\cite{lin2025discontinuous} extended the discontinuous ELM method by advancing the solution in time with an implicit--explicit discretization scheme. PF-PINNs~\cite{qiu2022physics} were developed for two-dimensional immiscible incompressible two-phase flows.}

{
	In this work, we develop an improved PINN method, termed the Extended Interface PINN (XI-PINN), for solving moving interface problems. The method is built upon the extended variable technique (EVT)~\cite{tseng2023cusp, hu2022discontinuity}. Compared with existing numerical methods for moving interface problems, this approach offers the following main advantages:
	\begin{itemize}
		\item[(1)] As a meshfree method, it avoids the difficulties that traditional grid-based methods face when dealing with high-dimensional problems or large deformations of the interface geometry.
		
		\item[(2)] Unlike existing multi-network approaches~\cite{jagtap2020extended, wang2021deep}, XI-PINN treats both subdomains with a single neural network. This means the network can adaptively distribute its parameters according to the complexity of the solution in each region, leading to a more uniform error distribution across subdomains. In contrast, multi-network methods that assign independent networks to different subdomains may suffer from unsuitable a priori parameter allocation, which can degrade the accuracy of the coupled solution for interface problems.
		
		\item[(3)] By designing the extension variable, XI-PINN can accurately capture various types of interface conditions, circumventing the limitations inherent in diffuse interface models.
		
		\item[(4)] When the solution exhibits no jump across the interface, XI-PINN can exactly enforce the zero-jump condition~\cite{li2025continuity}. This is particularly important for two-phase flow problems, as it helps preserve the correct physical behavior.
	\end{itemize}
	
	One limitation of the EVT is that it presupposes an explicitly given level set function to describe the interface. In practice, obtaining this function typically requires solving a level set evolution equation, which is a nontrivial task when an analytical expression is unavailable. To overcome this difficulty, we represent the level set function using an additional neural network. By employing a tailored temporal training strategy, this network can accurately track the moving interface, even under large deformations.
	
	We further present a comprehensive error analysis for the XI-PINN framework. Following recent developments \cite{jiao2025drm, ma2020machine}, we decompose the total error into three distinct components: the approximation error associated with the neural network space, the statistical error, and the optimization error. The bounds for the approximation error are guaranteed by the universal approximation properties of deep neural networks \cite{guhring2021approximation}, while the optimization error vanishes in the over-parameterized regime, as established in the literature \cite{liu2022loss}. Consequently, our primary theoretical contribution lies in providing a detailed proof of the statistical error specifically tailored to parabolic moving interface problems, thereby completing the overarching analytical framework. 
	
	We evaluate the performance of XI-PINN through a series of numerical experiments, including high-dimensional problems, the Oseen equations, and test cases characterized by severe interface deformations. Leveraging second-order optimization algorithms, XI-PINN delivers highly accurate solutions. In cases where the exact solution is continuous across the interface, or when employing a small network architecture, XI-PINN exhibits better accuracy compared to multi-network approaches. Furthermore, for a benchmark problem with the interface condition $h_D=0$, we show the spectrum of the neural tangent kernel (NTK). The results indicate that XI-PINN yields a more favorable spectral distribution than the Vanilla-PINN, demonstrating that within a single-network framework, XI-PINN effectively mitigates the spectral bias of the Vanilla-PINN and successfully captures sharp interface features.
}

The rest of this paper is organized as follows. In Section~\ref{sec2}, we first present the fundamental framework of Deep Neural Networks (DNNs), then introduce the XI-PINN method and propose our novel approach for solving the level set function. In Section~\ref{sec3}, we conduct a detailed error analysis of the XI-PINN method. Numerical results are presented in Section~\ref{sec4} to demonstrate the accuracy and efficacy of XI-PINN. Finally, we give some concluding remarks in Section~\ref{sec5}.

\section{Extended Interface PINN Method}\label{sec2}
In this section, we first review preliminaries on DNNs and the EVT. Subsequently, we introduce the XI-PINN method for solving parabolic interface problems. This method efficiently handles problems with geometrically complex interfaces and can be seamlessly extended to higher-dimensional scenarios.

\subsection{DNNs} \label{sec2.1}
We consider a standard fully connected feedforward neural network consisting of multiple linear transformations and nonlinear activation functions. 
Let the network have depth $L \ge 2$, $L \in \mathbb{Z}$, and let $\{ n_l \}_{l=0}^L$ denote the layer dimensions, where $n_0$ is the input dimension and $n_L$ is the output dimension. For each layer $l = 1, \dots, L$, define an affine transformation $\mathbf{T}^l: \mathbb{R}^{n_{l-1}} \to \mathbb{R}^{n_l}$ as follows:
\begin{equation*}
	\mathbf{T}^l(\mathbf{X}^{l-1}):= \mathbf{A}^l \mathbf{X}^{l-1} + \mathbf{b}^l,
\end{equation*}
where $\mathbf{A}^l \in \mathbb{R}^{n_l \times n_{l-1}}$ is the weight matrix and $\mathbf{b}^l \in \mathbb{R}^{n_l}$ is the bias vector. The $l$-th hidden layer ($l = 1, \dots, L-1$) applies an element-wise nonlinear activation function $\sigma: \mathbb{R} \to \mathbb{R}$:
\begin{equation*}
	\mathbf{X}^{l} = f^{l}(\mathbf{X}^{l-1}) := \sigma\bigl( \mathbf{T}^l( \mathbf{X}^{l-1} ) \bigr).
\end{equation*}
The final layer ($l = L$) produces the output via a purely affine transformation without activation, i.e., $\mathbf{X}^L = \mathbf{T}^L ( \mathbf{X}^{L-1})$.
The complete $L$-layer DNN is then the composition
\begin{equation*}
	\mathcal{NN}(\bm{x}; \theta) := \bigl( \mathbf{T}^L \circ f^{L-1} \circ \cdots \circ f^1 \bigr) (\mathbf{X}^0), \quad \mathbf{X}^0 = \bm{x},
\end{equation*}
where $\bm{x} \in \mathbb{R}^{n_0}$ is the input of the neural network, and the set of trainable parameters $\{ \mathbf{A}^l, \mathbf{b}^l\}_{l=1}^L$ is collected into a parameter vector $\theta$. The integers $L$ and $W:=\max \{ n_l,\ l=1,\dots,L\}$ are called the depth and width of the DNN, respectively.

Following the definition in~\cite{hu2024solving}, we denote by $\mathcal{H}_\sigma(L,N_\theta,B_\theta)$ the collection of DNN functions of depth $L$, with at most $N_\theta$ nonzero trainable parameters, where each parameter is bounded in absolute value by $B_\theta$, i.e.,
\begin{align*}
	\mathcal{H}_\sigma( L,N_\theta,B_\theta) := \bigl\{ \text{$L$-layer } \mathcal{NN}(\bm{x}; \theta): |\theta|_{\ell^0} \le N_\theta,\ |\theta|_{\ell^\infty} \le B_\theta \bigr\}.
\end{align*}
Here $|\cdot|_{\ell^0}$ and $|\cdot|_{\ell^\infty}$ denote the number of nonzero entries and the maximum norm of a vector, respectively. For brevity, we often write $\mathcal{H}$ to mean $\mathcal{H}_\sigma(L,N_\theta,B_\theta)$.

\subsection{EVT}\label{sec2.2}
The EVT has demonstrated remarkable success in solving elliptic interface problems, enabling highly accurate approximations of solutions with low regularity across interfaces using a single neural network (see~\cite{tseng2023cusp, hu2022discontinuity, ying2024accurate, bi2025xi}).

We first introduce two important functions: the level set function $\phi: \Omega \times [0,T_{\text{end}}] \to \mathbb{R}$ and the indicator function $\chi: \Omega \times [0,T_{\text{end}}] \to \mathbb{R}$. The level set function $\phi$ characterizing the interface satisfies~\cite{osher2004level}
\begin{equation*}
	\phi(\bm{x}, t) \begin{cases}
		<0, & \text{if } \bm{x} \in \Omega^-(t), \\
		=0, & \text{if } \bm{x} \in \Gamma(t), \\
		>0, & \text{if } \bm{x} \in \Omega^+(t),
	\end{cases}
\end{equation*}
and the indicator function is defined as
\begin{equation*}
	\chi(\bm{x}, t)= \begin{cases}
		-1, & \text{if } \bm{x} \in \Omega^-(t), \\ 
		1,  & \text{if } \bm{x} \in \Omega^+(t).
	\end{cases}
\end{equation*}
It is clear that the indicator function can be written in terms of the level set function. Consequently, we define an extension function $\widetilde{u}(\bm{x}, t, z)$ on $\widetilde{\Omega}:= \Omega \times [0, T_{\text{end}}] \times \mathbb{R} \subset \mathbb{R}^{d+2}$ such that
\begin{equation} \label{eq2.5}
	\widetilde{u}(\bm{x}, t, z(\bm{x}, t)) = u(\bm{x}, t), \quad \bm{x} \in \Omega, \quad t\in[0,T_{\text{end}}],
\end{equation}
where $z$ represents the extension variable. Depending on the interface jump condition $h_D$, the extension variable is defined as
\begin{equation}
	z( \bm{x}, t)  = \begin{cases}
		\chi( \bm{x}, t), & \text{if } h_D \neq 0,\\[4pt]
		\Phi( \bm{x}, t):= | \phi( \bm{x}, t) |, & \text{if } h_D = 0,
	\end{cases}
\end{equation}
As demonstrated in~\cite{li2025continuity}, setting $z=\Phi(\bm{x}, t)$ for the case $h_D=0$ improves approximation accuracy, since the neural network component $u_{\mathcal{NN}} \in \mathcal{H}$ naturally enforces continuity across the entire domain $\Omega$. It can be observed that, by the definitions of $\widetilde{u}$ and the extended variable $z(\bm{x}, t)$, $\widetilde{u}$ naturally exhibits discontinuities or jumps in its derivatives at the interface $\Gamma(t)$.
\subsection{XI-PINN}
As discussed in Section~\ref{sec2.2}, although $u$ (or its derivative) is discontinuous across the interface, the EVT lifts it to a continuous function $\widetilde u$ defined on $\widetilde{\Omega} \subset \mathbb{R}^{d+2}$. We therefore approximate the solution of Eqs.~\eqref{eq1.1}--\eqref{eq1.7} by a DNN $\widetilde u_\theta\in\mathcal{H}$ with input dimension $d+2$, trained through the following discrete optimization problem:
\begin{equation}\label{eq2.7}
	\operatorname*{arg\,min}_{\widetilde{u}_{\theta} \in \mathcal{H}} \ell = \sum_{i=1}^{5} \ell_i(\widetilde{u}_{\theta}),
\end{equation}
with
{
\begin{align}
	\ell_1(\widetilde{u}_{\theta})  &= \dfrac{1}{N_\Omega} \sum_{\bm{s}_j \in \mathcal{S}_{\Omega}} \bigl| \partial_t \widetilde{u}_{\theta}(\bm{s}_j, z(\bm{s}_j))  -\nabla \cdot \bigl( \beta \nabla \widetilde{u}_{\theta}(\bm{s}_j, z(\bm{s}_j)) \bigr) - f(\bm{s}_j) \bigr|^2,  \label{eq2.8} \\
	\ell_2(\widetilde{u}_{\theta}) &= \dfrac{1}{N_{\partial \Omega}} \sum_{\bm{s}_j \in \mathcal{S}_{\partial \Omega}} \bigl| \widetilde{u}_{\theta}(\bm{s}_j, z(\bm{s}_j)) - g(\bm{s}_j) \bigr|^2,  \label{eq2.9} \\
	\ell_3(\widetilde{u}_{\theta}) &= \dfrac{1}{N_{\Omega_0}} \sum_{\bm{s}_j \in \mathcal{S}_{\Omega_0}} \bigl| \widetilde{u}_{\theta}(\bm{s}_j, z(\bm{s}_j)) - u_0(\bm{s}_j) \bigr|^2,  \label{eq2.10}\\
	\ell_4(\widetilde{u}_{\theta}) &= \dfrac{1}{N_{\Gamma}} \sum_{\bm{s}_j \in \mathcal{S}_{\Gamma}} \bigl| \bigl[ \beta \nabla \widetilde{u}_{\theta}(\bm{s}_j, z(\bm{s}_j)) \cdot \mathbf{n} \bigr]_{\Gamma(t)} - h_N(\bm{s}_j) \bigr|^2,  \label{eq2.12}\\
	\ell_5(\widetilde{u}_{\theta}) &= \dfrac{1}{N_{\Gamma}} \sum_{\bm{s}_j \in \mathcal{S}_{\Gamma}} \bigl| \bigl[ \widetilde{u}_{\theta}(\bm{s}_j, z(\bm{s}_j)) \bigr]_{\Gamma(t)} - h_D(\bm{s}_j) \bigr|^2, \label{eq2.11}
\end{align}
where $\mathcal{S}_\Omega$, $\mathcal{S}_{\partial\Omega}$, $\mathcal{S}_{\Omega_0}$, and $\mathcal{S}_\Gamma$ denote the sets of uniformly sampled collocation points within the respective domains, with $N_\Omega, N_{\partial \Omega}, N_{\Omega_0}$, and $N_{\Gamma}$ denoting the total number of points in each corresponding set.}

\begin{remark}
	When the interface jump condition satisfies $h_D=0$, the extension variable $z$ can be set to $z({\bm{x}},t) = \Phi({\bm{x}},t)$. Consequently, the $\ell_5$ term can be omitted from the optimization formulation~(\ref{eq2.7}). 
\end{remark}

\begin{remark} \label{remark2}
	For the derivative components within the discrete loss formulations (\ref{eq2.8}) and (\ref{eq2.12}), applying the chain rule yields the following three identities:
	\begin{itemize}
		\item[(1)] $\partial_t \widetilde{u}_\theta = D_t \widetilde{u}_\theta + D_z \widetilde{u}_\theta \, \partial_t z$;
		\item[(2)] $\nabla \widetilde{u}_\theta = \nabla_{\bm{x}}\widetilde{u}_\theta + D_z \widetilde{u}_\theta \, \nabla z$;
		\item[(3)] $\nabla \cdot \nabla \widetilde{u}_\theta = \Delta_{\bm{x}} \widetilde{u}_\theta + 2 \nabla z \cdot \nabla_{\bm{x}}( D_z \widetilde{u}_\theta) + |\nabla z|^2 D_z^2 \widetilde{u}_\theta + D_z \widetilde{u}_\theta \, \Delta z$.
	\end{itemize}
	Here, $D_t$ and $D_z$ denote the partial derivatives of $\widetilde{u}_\theta$ with respect to time $t$ and the extended variable $z$, respectively. Moreover, $\nabla_{\bm{x}} \widetilde{u}_\theta \in \mathbb{R}^{d}$ represents the spatial gradient of $\widetilde{u}_\theta$, and $\Delta_{\bm{x}}$ denotes the spatial Laplacian.
	
	{When $z = \chi(\bm{x},t)$ (i.e., $h_D \neq 0$), the derivatives of $z$ vanish almost everywhere; thus, the smoothness of the interface does not affect the numerical method. In contrast, setting $z = \Phi(\bm{x},t)$ (i.e., $h_D = 0$) requires the underlying level set function $\phi(\bm{x},t)$ to be twice differentiable with respect to $\bm{x}$ and differentiable with respect to $t$. Because the regularity of the interface $\Gamma(t)$ directly governs the regularity of the level set function, we assume $\Gamma(t)$ to be of class $C^2$ whenever $z = \Phi(\bm{x},t)$ is employed. Moreover, if the interface lacks smoothness, the exact solution $u$ may develop singularities at corners in the subdomains $\Omega^{\pm}$. Solving PDEs with such geometric singularities naturally poses well-known challenges for the PINNs \cite{hu2024solving}.}
\end{remark}

\subsection{Construction of the Level Set Function} \label{sec2.4}
The level set function plays a pivotal role in the EVT, as it determines the location of points relative to the subdomains, thereby defining the indicator function. Furthermore, when $h_D=0$, it is used to directly construct the extended variable. Consequently, obtaining the time-evolving level set function is essential for the proposed framework. Assuming the velocity field $\mathcal{V}$ is continuously extended to the entire domain $\Omega$ such that the trajectory of every point $\bm{x} \in \Omega$ is governed by the ODE (\ref{eq1.7}), the level set function $\phi(\bm{x}, t)$ satisfies the following advection equation:
\begin{align} \label{advection equation}
	\frac{\partial \phi(\bm{x}, t)}{\partial t} + \mathcal{V} \cdot \nabla \phi(\bm{x}, t) &= 0, \quad \bm{x} \in \Omega,\; t \in [0, T_{\text{end}}],\\
	\phi(\bm{x}, 0) &= \phi_0(\bm{x}), \quad \bm{x} \in \Omega,
\end{align}
where $\phi_0(\bm{x})$ denotes the prescribed initial level set function at $t=0$. The governing equation may be solved using the method of characteristics. For any point $\bm{x}_0 \in \Omega$ at $t=0$, provided that its position at $t=T$ admits an explicit representation $\bm{x}_T = \bm{x}_0 + \mathcal{F}(\bm{x}_0, T)$, where $\mathcal{F}$ is the displacement function (i.e., $\mathcal{F}(\bm{x}_0, T) = \bm{x}_T - \bm{x}_0$), the level set function can be constructed as
\begin{equation*}
	\phi(\bm{x}, T) = \phi_0\bigl( \bm{x} - \mathcal{F}(\bm{x}, T) \bigr). 
\end{equation*}
Alternatively, to bypass solving the evolution equation entirely, we can employ a neural network to directly construct an appropriate level set function. For a given reference time $t_0 \in [0,T_{\text{end}}]$, we define the flow map as follows:
\begin{equation} \label{eq2.16}
	\mathbf{X}: \Gamma(t_0) \times [t_0, T_{\text{end}}] \to \bigcup_{t\in[t_0,T_{\text{end}}]}\Gamma(t), \quad \bm{x}_t = \mathbf{X}(\bm{x}, t; t_0), \quad t\in[t_0,T_{\text{end}}],
\end{equation}
where $\bm{x} \in \Gamma(t_0)$ and $\bm{x}_t \in \Gamma(t)$. The flow map can be obtained using the solution of the initial-value problem (\ref{eq1.7}):
\begin{equation*}
	\frac{\mathrm{d}}{\mathrm{d}t} \mathbf{X}(\bm{x},t; t_0) = \mathcal{V}\bigl( \mathbf{X}(\bm{x},t; t_0), t \bigr), \quad \mathbf{X}(\bm{x},t_0; t_0) = \bm{x}, \quad \bm{x} \in \Gamma(t_0).
\end{equation*}
The inverse of $\mathbf{X}(\bm{x},t; t_0)$ is denoted by $\widetilde{\mathbf{X}}(\bm{x}, t; t_0): \Gamma(t) \times [t_0 ,T_{\text{end}}] \to \Gamma(t_0)$. We extend the mapping $\widetilde{\mathbf{X}}$ to the whole domain and define the extension mapping
\[
\widehat{\mathbf{X}} : \Omega \times [t_0,T_{\text{end}}] \to \mathbb{R}^d,
\]
which is required to satisfy the following properties:
\begin{equation} \label{flow map properties}
	\small
	\widehat{\mathbf{X}}(\mathbf{x},t;t_0)|_{\mathbf{x} \in \Gamma(t)} =  \widetilde{\mathbf{X}}(\mathbf{x},t;t_0) \text{ and }
	\widehat{\mathbf{X}}(\mathbf{x},t;t_0) \text{ is orientation-preserving and injective.}
\end{equation}
Therefore, the level set function at time $t$ can be defined as
\begin{equation} \label{eq2.20}
	\phi( \bm{x}, t) = \phi_0\bigl( \widehat{\mathbf{X}}(\bm{x},t;0) \bigr).
\end{equation}
The gradient of this function can be computed via the chain rule:
\begin{equation*}
	\nabla_{\bm{x}} \phi(\bm{x},t) = \nabla_{\widehat{\mathbf{X}}} \phi_0\bigl( \widehat{\mathbf{X}} \bigr) \, \nabla_{\bm{x}} \widehat{\mathbf{X}}(\bm{x},t;0),
\end{equation*}
where $\nabla_{\widehat{\mathbf{X}}} \phi_0$ is the gradient vector and $\nabla_{\bm{x}} \widehat{\mathbf{X}}$ is the Jacobian matrix.

Next, we employ neural networks to approximate the mapping $\widehat{\mathbf{X}}$. Denote by $\mathcal{F}_\theta$ the neural network function, and define the approximating function for $\widehat{\mathbf{X}}$ as
\begin{equation} \label{eq2.21}
	\widehat{\mathbf{X}}_\theta( \bm{x}, t; 0) = \mathcal{F}_\theta( \bm{x}, t) + \bm{x}. 
\end{equation}
Let $\mathbf{X}_{\mathrm{RK}}$ be the approximation of $\mathbf{X}$ computed using the 4th-order explicit Runge--Kutta method~\cite{verner1978explicit}. We then determine the approximate displacement function $\mathcal{F}_\theta$ and the resulting flow map $\widehat{\mathbf{X}}_\theta$ by minimizing the following loss function:
\begin{equation*}
	\ell_{\mathrm{LSF}} = \frac{1}{N_{\Omega}} \sum_{i=1}^{N_{\Omega}} \bigl| \mathcal{F}_\theta( \bm{x}_{i, \Omega}, 0) \bigr|^2 + \frac{1}{N_t N_\Gamma} \sum_{j=1}^{N_t} \sum_{i=1}^{N_\Gamma} \bigl| \widehat{\mathbf{X}}_\theta\bigl( \mathbf{X}_{\mathrm{RK}}(\bm{x}_{i,\Gamma}, t_j; 0), t_j; 0 \bigr) - \bm{x}_{i,\Gamma} \bigr|^2,
\end{equation*}
where $\{ \bm{x}_{i,\Omega} \}_{i=1}^{N_{\Omega}}$ and $\{ \bm{x}_{i,\Gamma} \}_{i=1}^{N_{\Gamma}}$ are training points sampled from the domain $\Omega$ and the initial interface $\Gamma(0)$, respectively, and $\{ t_j \}_{j=1}^{N_{t}}$ forms a sufficiently fine temporal discretization of $[0,T_{\text{end}}]$. 

Notably, we require $\widehat{\mathbf{X}}_\theta$ to satisfy the properties \eqref{flow map properties}. If the approximate flow map $\widehat{\mathbf{X}}_\theta$ fails to be injective or orientation-preserving, the level set function  
\[
\phi_\theta(\bm{x}, t) := \phi_0\bigl( \widehat{\mathbf{X}}_\theta(\bm{x}, t; 0) \bigr)
\]
loses its validity as an implicit interface representation. Lack of injectivity produces global folding: two distinct points $\bm{x}_1 \neq \bm{x}_2$ in the current domain can be pulled back to the same reference point $\bm{y}$, forcing $\phi_\theta(\bm{x}_1, t) = \phi_\theta(\bm{x}_2, t) = \phi_0(\bm{y})$. If $\bm{y}$ lies on the initial interface, the zero level set spuriously bifurcates, creating non-physical internal boundaries and corrupting the subdomain indicator function. Violation of orientation preservation, on the other hand, leads to local Jacobian degeneracy or sign reversal; the map can then erroneously assign points that belong to $\Omega^+(t)$ to $\Omega^-(0)$, destroying the inside/outside sign convention. These defects are the direct analogues of mesh entanglement and element inversion in finite element computations.

To mitigate this issue, we adopt a temporal partitioning strategy, dividing $[0, T_{\text{end}}]$ into smaller sub-intervals:
\[
[0, T_{\text{end}}] = \bigcup_{k=1}^{K}[T_{k-1}, T_k], \quad T_0 = 0,\; T_K = T_{\text{end}},
\]
where the set $\{ T_k \}_{k=1}^{K} \subseteq \{ t_j \}_{j=1}^{N_t}$ can be adaptively chosen by monitoring the Jacobian determinant of $\widehat{\mathbf{X}}_\theta$. For any $t$ in the $k$-th sub-interval $[T_{k-1}, T_k]$, we can decompose $\widehat{\mathbf{X}}(\bm{x}, t; 0)$ into a composition of sub-mappings:
\begin{equation} \label{eq2.24}
	\widehat{\mathbf{X}}(\bm{x}, t;0) = \bigl[ \widehat{\mathbf{X}}(\cdot, t;T_{k-1}) \circ \widehat{\mathbf{X}}(\cdot, T_{k-1}; T_{k-2}) \circ \cdots \circ \widehat{\mathbf{X}}(\cdot, T_{1}; 0) \bigr](\bm{x}).
\end{equation}
For each short time sub-interval $[T_{k-1}, T_k]$, we compute a high-precision approximate sub-mapping $\widehat{\mathbf{X}}_{\theta, k}$ using a small neural network. Note that employing an excessively large network may lead to overfitting, which in turn could introduce undesirable mapping entanglement.

Moreover, the displacement of interface points, $\Delta \bm{x}$, remains moderate over short time intervals, effectively preventing mesh entanglement, particularly for problems involving large interface deformations. Correspondingly, we employ a series of sub-networks $\widehat{\mathbf{X}}_{\theta, k}$ to approximate $\widehat{\mathbf{X}}(\cdot, t; T_{k-1})$ for $t \in [T_{k-1}, T_k]$, with the loss function adjusted as follows:
\begin{equation}\label{eq2.25}
	\begin{aligned}
		\ell_{\mathrm{LSF}}^k = &\frac{1}{N_{\Omega}} \sum_{i=1}^{N_{\Omega}} \bigl| \mathcal{F}_\theta(\bm{x}_{i, \Omega}, T_{k-1}) \bigr|^2 + \frac{1}{N_{t,k} N_{\Gamma}} \sum_{j=1}^{N_{t,k}} \sum_{i=1}^{N_{\Gamma}} \\
		& \qquad \bigl| \widehat{\mathbf{X}}_\theta\bigl( \mathbf{X}_{\mathrm{RK}}(\bm{x}_{i,\Gamma}, t_j^k; 0) , t_j^k; T_{k-1} \bigr) - \mathbf{X}_{\mathrm{RK}}(\bm{x}_{i,\Gamma}, T_{k-1}; 0) \bigr|^2,
	\end{aligned}
\end{equation}
where $\{ t_j^k \}_{j=1}^{N_{t,k}}$ denotes a discrete time series within $[T_{k-1}, T_k]$, sampled from the global set $\{ t_j \}_{j=1}^{N_t}$.

To detect potential mesh entanglement (indicated by negative Jacobian determinants), we locally linearize the mapping via a first-order Taylor expansion. This approach employs uniform sampling over $\Omega$ with a sufficiently small grid spacing $h$, which ensures the validity of the local linear approximation:
\begin{equation} \label{eq2.26}
	\widehat{\mathbf{X}}_{\theta, k}(\bm{x},t;T_{k-1}) \approx \widehat{\mathbf{X}}_{\theta, k}(\mathbf{x}_0,t;T_{k-1}) + \mathbf{J}(\bm{x}_0)(\bm{x} - \bm{x}_0), \quad \forall \bm{x} \in B_\varepsilon(\bm{x}_0),
\end{equation}
where $\mathbf{J}(\bm{x}_0):=\nabla_{\bm{x}} \widehat{\mathbf{X}}_{\theta,k}$ denotes the Jacobian matrix of the mapping and $B_\varepsilon \subset \Omega$ is a small neighborhood of the sample point $\bm{x}_0 \in \Omega$. A negative Jacobian determinant ($\det(\mathbf{J}) < 0$) indicates that the mapping fails to preserve orientation. It is worth noting that, for a smooth map between two compact simply connected domains, orientation preservation (positive Jacobian determinant everywhere) guarantees global injectivity. To enforce this condition numerically, we impose a positivity constraint by maintaining $\det(\mathbf{J}) > \delta$, where $\delta>0$ is a prescribed threshold. A Jacobian determinant falling below $\delta$ (i.e., $\det(\nabla_{\bm{x}} \widehat{\mathbf{X}}_{\theta,k}) < \delta$) serves as a topological early warning signal. At this critical point, we initialize a new neural network $\widehat{\mathbf{X}}_{\theta,k+1}$ to learn the deformation mapping for the subsequent interval $[T_k, T_{k+1}]$.

We therefore develop an adaptive time-stepping algorithm, which consists of (1) a sequence of sub-mappings $\{ \widehat{\mathbf{X}}_{\theta, k} \}_{k=1}^{K}$ and (2) a corresponding partition of the time domain into sub-intervals $\{ [T_{k-1}, T_k] \}_{k=1}^{K}$ based on the discrete temporal grid $\{ t_j \}_{j=1}^{N_t}$ of $[0,T_{\text{end}}]$.
This adaptive network growth strategy ensures that $\phi_\theta$ remains valid over the entire time domain. We describe the complete time-adaptive training procedure in Algorithm~\ref{algorithm 1} as follows:

\begin{algorithm}
\caption{Time-Stepping Neural Network Training for (\ref{eq2.24})}
	\label{algorithm 1}
	\SetAlgoLined
	\DontPrintSemicolon
	\KwIn{Initial time $t_0=0$, threshold $\delta > 0$; Time sequence $\{t_i\}_{i=0}^{N_t}$, sample points $\{{\bm{x}}_{i,\Omega}\}_{i=1}^{N_{\Omega}}$ and $\left\lbrace {\bm{x}}_{i,\Gamma} \right\rbrace_{i=1}^{N_{\Gamma}}$}
	\KwOut{Trained neural networks $\widehat{\mathbf{X}}_{\theta,k}$ for each adaptive interval $[T_{k-1}, T_k]$, $k=1,\cdots,K$}
	
	Set network index $k \leftarrow 1$ and current time index $i \leftarrow 0$\; 
	
	Initialize neural network $\widehat{\mathbf{X}}_{\theta,k}$\;
	
	Set current start time $t_s \leftarrow t_i$, $T_0 \leftarrow t_i$\;
	
	\While{$i \le N_t$}{
		$i \leftarrow i + 1$\;
		
		Train $\widehat{\mathbf{X}}_{\theta,k}$ on time steps $\{t_s, \dots, t_i\}$ using loss function (\ref{eq2.25})\;
		
		Compute Jacobian determinant $|\mathbf{J}|$ of $\widehat{\mathbf{X}}_{\theta,k}$\;
		
		\eIf{$|\mathbf{J}| > \delta$ \textbf{and} $i \le N_t$}{
			Continue to next time step \;
		}{
			\If{$|\mathbf{J}| \leq \delta$}{
				\textbf{Terminate training} of $\widehat{\mathbf{X}}_{\theta,k}$ for current interval $[t_s, t_i]$\;
				
				Save the sub-network $\widehat{\mathbf{X}}_{\theta,k}$\;
				
				$T_k \leftarrow t_i$\;
			}
			$k \leftarrow k + 1$ \tcp*{Move to next network}
			Initialize new network $\widehat{\mathbf{X}}_{\theta,k}$\;
			
			$t_s \leftarrow t_i$ \tcp*{Reset start time to current endpoint}
		}
	}
	$K \leftarrow k, T_K \leftarrow T_{\text{end}}$\;
\end{algorithm}

\begin{remark}
	In Algorithm~\ref{algorithm 1}, we assume that the sampling density satisfies the local linearization validity condition. Here $\det(\mathbf{J})$ denotes the minimum Jacobian determinant computed over all uniformly sampled points in $\Omega$.
\end{remark}
\begin{remark}
	To guarantee the mapping accuracy on the interface, we use only the interface points as training data for the neural network. Moreover, we can enforce the identity mapping at the beginning of each sub-interval by setting
	\begin{equation*}
		\widehat{\mathbf{X}}_{\theta, k}(\bm{x}, t; 0) = (t - T_{k-1})\, \mathcal{F}_\theta(\bm{x}, t) + \bm{x},
	\end{equation*}
	so that $\widehat{\mathbf{X}}_{\theta, k}(\bm{x}, T_{k-1}; 0) = \bm{x}$.
\end{remark}
\section{Error analysis} \label{sec3}

In this section, we present the error analysis of the XI-PINN framework introduced in Section~\ref{sec2}. We assume that the domain $\Omega \subseteq \mathbb{R}^d$ is bounded with a Lipschitz boundary $\partial \Omega$, and $\Gamma(t)$ is a closed smooth curve (or surface) within $\Omega$. To simplify the analysis, we assume $h_D = 0$. This implies that the neural network approximation $\widetilde{u}_{\mathcal{NN}} \in \mathcal{H}$ remains continuous across the interface.

The space-time domain is denoted by $Q:=\Omega \times [0,T_{\text{end}}]$ and $\Sigma:= \partial \Omega \times [0,T_{\text{end}}]$. The extended space-time domain is defined as 
\[
\widetilde{Q}:= Q \times \bigl[ \min_{(\bm{x}, t)\in Q} z(\bm{x}, t), \max_{(\bm{x}, t)\in Q} z(\bm{x}, t) \bigr].
\]
For each subdomain $\omega \subseteq \Omega$, we denote by $H^k(\omega)$ the standard Sobolev space equipped with the norm $\|\cdot\|_{H^k(\omega)}$, and define the time-dependent Bochner space $H^l(0, T; H^k(\omega))$ as the space of all measurable functions $u: [0,T] \to H^k(\omega)$ with the norm $\|\cdot\|_{H^l(0, T; H^k(\omega))}$. Here, we let $H^k(\Omega^+(t) \cup \Omega^-(t))$ denote the space of functions that belong to $H^k$ on each subdomain separately, equipped with the piecewise norm
\[
\|v\|_{H^k(\Omega^+(t) \cup \Omega^-(t))}:= \|v\|_{H^k(\Omega^+(t))} + \|v\|_{H^k(\Omega^-(t))}.
\]
We also employ the Besov space $B_{2,1}^{1/2}(\omega):=\bigl( L^2(\omega), H^1(\omega) \bigr)_{1/2, 1}$ obtained by real interpolation (see~\cite{tartar2007introduction}). The piecewise Besov space $B_{2,1}^{1/2}(\Omega^+(t) \cup \Omega^-(t))$ and its norm $\|v\|_{B_{2,1}^{1/2}(\Omega^+(t) \cup \Omega^-(t))}$ are defined analogously. Moreover, we have the embedding $H^{1}(\omega) \hookrightarrow B_{2,1}^{1/2}(\omega)$. Let $r$ and $s$ be two non-negative real numbers, and define the space $H^{r,s}(Q) := L^2(0,T; H^{r}(\Omega)) \cap H^s(0, T; L^2(\Omega))$ with the norm $\|\cdot\|_{H^{r,s}}$; this is a Hilbert space (see~\cite{Non1972}). For a function $v(t)$, we write $v' := \frac{\mathrm{d}}{\mathrm{d} t} v$ for its time derivative (or $\partial_t v$ if $v$ depends on space as well) and let $v_0:=v(0)$ denote its initial value. For simplicity, we use the shorthand notation $A \lesssim B$ for the inequality $A \le C B$, where $C>0$ is a generic constant independent of $N_\Omega$, $N_{\partial \Omega}$, $N_{\Omega_0}$, and $N_\Gamma$, the numbers of training points specified in the loss functions (\ref{eq2.8})--(\ref{eq2.12}).

\subsection{Error analysis of the XI-PINN method}
We first establish an energy estimate for the solution $u$ to Eqs. (\ref{eq1.1})-(\ref{eq1.7}).
\begin{lemma}\label{lemma1}
Assume $h_N \in L^2\left( 0, T_{\text{end}}; L^2(\Gamma(t))\right)$,  $u_0 \in H^1(\Omega)$, $g \in H^{2, 1}(\Sigma)$, with $u_0$ and $g$ satisfying the compatibility conditions. Then there exists a unique solution $u \in L^2\left( 0, T; H^1(\Omega)\right)$ to (\ref{eq1.1})-(\ref{eq1.7}) satisfying the estimate:
	\begin{equation}
		\begin{aligned}
			\|u\|_{L^2\left( 0, T; H^1(\Omega)\right) }^2 \lesssim& \|f\|_{L^2(0,T_{\text{end}};L^2(\Omega))}^2 + \|u_0\|_{H^1(\Omega)}^2 + \|h_N\|_{L^2\left( 0, T_{\text{end}}; L^2(\Gamma(t))\right) }^2 \\
			&+ \|g\|_{L^2\left( 0, T_{\text{end}}, H^2(\partial \Omega)\right) }^2 + \|g'\|_{L^2\left( 0, T_{\text{end}}, L^2(\partial \Omega)\right)}^2.
		\end{aligned}
	\end{equation}
\end{lemma}
\begin{proof} 
	The first step involves homogenizing the boundary conditions. Invoking the Trace Theorem\cite{Non1972}, there exists $w \in H^{2,1}(Q)$ satisfying
	\begin{equation*}
		\left. w({\bm{x}}, t)\right|_{\Sigma} = g({\bm{x}},t), \quad
		\left. w({\bm{x}}, 0)\right|_{\Omega} = u_0({\bm{x}}),
	\end{equation*}
	and
	\begin{equation} \label{eq3.2}
		\|w\|_{H^{2,1}(Q)} \lesssim \|g\|_{H^{2, 1}(\Sigma)} + \|u_0\|_{H^1(\Omega)}.
	\end{equation}
	Setting $v=u-w$, we reformulate problem (\ref{eq1.1})-(\ref{eq1.5}) as
	\begin{equation}\label{eq3.3}
	\left\{	\begin{aligned}
			\partial_t v - \nabla \cdot \left( \beta \nabla v\right)  &= \widetilde{f} \quad \mathrm{in}\; \Omega = \Omega^+(t) \cup \Omega^-(t), \quad  t \in [0, T_{\text{end}}], \\
			v\left( \cdot, t\right) &= 0 \quad \mathrm{on}\; \partial \Omega, \quad t \in [0, T_{\text{end}}],\\
			v\left( \cdot, 0\right) &= 0 \quad \mathrm{in}\; \Omega = \Omega^+(0) \cup \Omega^-(0),
		\end{aligned}\right.
	\end{equation}
	with the interface condition
	\begin{equation}\label{eq3.7}
		\begin{aligned}
		[u]_{\Gamma(t)} &= 0, \quad t\in [0,T_{\text{end}}],\\
		[\beta \nabla v \cdot \mathbf{n}]_{\Gamma(t)}& = \widetilde{h}_N, \quad t\in [0,T_{\text{end}}], 			
		\end{aligned}
	\end{equation}
	where $\widetilde{f} := f - w' + \nabla \cdot (\beta \nabla w) \in L^2(0, T_{\text{end}}; L^2(\Omega))$, and $\widetilde{h}_N:=h_N - [\beta \nabla w \cdot \mathbf{n}]_{\Gamma(t)} \in L^2(0, T_{\text{end}}; L^2(\Gamma(t)))$. 
	{The existence of a weak solution \( v \in L^2(0,T_{\text{end}}; H_0^1(\Omega)) \) with \( v' \in L^2(0,T_{\text{end}}; H^{-1}(\Omega)) \) to (\ref{eq3.3})-(\ref{eq3.7}) follows by the Galerkin approximation method (cf.~\cite[Section 7.1]{evans2022partial}). 
	Then, by combining standard energy estimates for parabolic equations\cite{evans2022partial} with the triangle inequality, we obtain
	\begin{equation*}
		\begin{aligned}
			\|u\|_{L^2(0, T_{\text{end}}; H^1(\Omega))}^2 
			&\lesssim \|v\|_{L^2(0, T_{\text{end}}; H^1(\Omega))}^2  + \|w\|_{L^2(0, T; H^1(\Omega))}^2 \\
			&\lesssim \|v\|_{L^2(0, T_{\text{end}}; H^1(\Omega))}^2 +  \|g\|_{H^{2,1}(\Sigma)}^2 + \|u_0\|_{H^1(\Omega)},
		\end{aligned}
	\end{equation*}
	which combines the definition of norm $\|g\|_{H^{2,1}(\Sigma)}$  immediately yields the result.}
\end{proof}

For subsequent error analysis, we define the continuous loss function as:
\begin{equation} \label{eq3.12}
	{\mathcal{L}}(\widetilde{u}_{\theta}) = \mathop{\sum}_{i=1}^4 {\mathcal{L}_i}(\widetilde{u}_{\theta}), 
\end{equation}
\begin{align*}
	&{\mathcal{L}_1}(\widetilde{u}_{\theta}) = \int_{0}^{T_{\text{end}}}\int_{\Omega} \left| \partial_t \widetilde{u}_{\theta}({\bm{x}}, t, z({\bm{x}}, t)) - \nabla \cdot (\beta \nabla\widetilde{u}_{\theta}({\bm{x}}, t, z({\bm{x}}, t))) - f({\bm{x}}, t) \right|^2 \mathrm{d}{\bm{x}} \mathrm{d}t, 
	\\
	&\begin{aligned}
		{\mathcal{L}_2}(\widetilde{u}_{\theta}) = &\int_{0}^{T_{\text{end}}}\int_{\partial \Omega} { \mathop{\sum}_{|\alpha|=0}^2 }\left| \partial^\alpha_{\bm{x}} \left( \widetilde{u}_{\theta}({\bm{x}}, t, z({\bm{x}}, t)) - g({\bm{x}}, t)\right)  \right|^2 \mathrm{d}{\bm{x}} \mathrm{d}t\\
		&+\int_{0}^{T_{\text{end}}}\int_{\partial \Omega} \left| \partial_t \widetilde{u}_{\theta}({\bm{x}}, t, z({\bm{x}}, t)) - \partial_t g({\bm{x}}, t) \right|^2 \mathrm{d}{\bm{x}} \mathrm{d}t 
	\end{aligned}
	\\
	&{\mathcal{L}_3}(\widetilde{u}_{\theta}) = \int_{\Omega} {\mathop{\sum}_{|\alpha| = 0}^1 }\left| \partial^\alpha_{\bm{x}} \left( \widetilde{u}_{\theta}({\bm{x}}, 0, z({\bm{x}}, 0)) - u_0({\bm{x}}) \right) \right|^2 \mathrm{d}{\bm{x}},
	\\
	&{\mathcal{L}_4}(\widetilde{u}_{\theta}) = \int_{0}^{T_{\text{end}}}\int_{\Gamma(t)} \left| [\beta \nabla \widetilde{u}_{\theta}({\bm{x}}, t, z({\bm{x}}, t)) \cdot \mathbf{n}]_{\Gamma(t)} - h_N({\bm{x}}, t)\right|^2 \mathrm{d}{\bm{x}} \mathrm{d}t,
\end{align*}
where {$\alpha = (\alpha_1, \alpha_2, \dots, \alpha_d) \in \left( \mathbb{N}\cup \left\lbrace 0\right\rbrace \right) ^d$} is a multi-index for spatial coordinates with $|\alpha| = \sum_{k=1}^d \alpha_k$, and $\partial_{{\bm{x}}}^\alpha = \dfrac{\partial^{|\alpha|}}{\partial x_1^{\alpha_1} \cdots \partial x_d^{\alpha_d}}$. Further, the corresponding discrete form of loss function \eqref{eq3.12} is defined by
\begin{equation} \label{eq3.17}
	 {\mathcal{L}}^d(\widetilde{u}_{\theta}) = \mathop{\sum}_{i=1}^4 \widehat{\ell}_i(\widetilde{u}_{\theta}),
\end{equation}
with $\widehat{\ell}_1(\widetilde{u}_{\theta})=\left|\Omega(t)\right| {\ell}_1(\widetilde{u}_{\theta})$, $\widehat{\ell}_4(\widetilde{u}_{\theta}) =\left|\Gamma(t)\right| {\ell}_4(\widetilde{u}_{\theta})$, and
\begin{align*}
	&\begin{aligned}
		\widehat{\ell}_2(\widetilde{u}_{\theta}) 
		&= \dfrac{\left| \partial \Omega(t)\right| }{N_{\partial \Omega}} \sum_{\bm{s}_j \in \mathcal{S}_{\partial \Omega}} \left[ { \mathop{\sum}_{|\alpha|=0}^2 }\left| \partial^\alpha_{\bm{x}} \left( \widetilde{u}_{\theta}(({\bm{s}}_j), z({\bm{s}}_j)) - g(\bm{s}_j)\right) \right|^2 \right.  \\
		&\left. \qquad \qquad \qquad \qquad + \left| \left( \partial_t \widetilde{u}_{\theta}((\bm{s}_j), z(\bm{s}_j)) -  \partial_t g(\bm{s}_j)\right)  \right|^2\right]  ,
	\end{aligned}
	\\
	&\widehat{\ell}_3(\widetilde{u}_{\theta}) = \dfrac{\left| \Omega(0) \right|}{N_{\Omega_0}} \mathop{\sum}_{\bm{s}_j \in \mathcal{S}_{\Omega_0}} { \mathop{\sum}_{|\alpha| = 0}^1 } \left| \partial^\alpha_{\bm{x}}\left(  \widetilde{u}_{\theta}\left(\bm{s}_j,z\left(\bm{s}_j\right) \right) - u_0\left(\bm{s}_j\right)\right) \right| ^2,
\end{align*}

Compared to (\ref{eq2.7})-(\ref{eq2.12}), the term \( \ell_5 \) has been eliminated by exploiting the level-set function's continuity, while derivative information has been incorporated into both the initial condition (\ref{eq2.10}) and boundary term (\ref{eq2.9}) inspired by the energy estimates. We remark that although the error analysis motivates including derivative terms in the loss function, our neural network implementation retains the original loss function defined in (\ref{eq2.7})-(\ref{eq2.12}).

Let $\widetilde{u}^*$ and $\widetilde{u}^*_d$ denote the best approximations in the function class $\mathcal{H}$ according to formulas (\ref{eq3.12}) and (\ref{eq3.17}), respectively. Namely,
\begin{equation*}
	\widetilde{u}^* := \mathop{\arg \min}_{v_\theta \in \mathcal{H}} {\mathcal{L}}(v_\theta), \quad \widetilde{u}_d^* := \mathop{\arg \min}_{v_\theta \in \mathcal{H}}  {\mathcal{L}}^d(v_\theta).
\end{equation*}
Further denote by $\widetilde{u}^*_{d\mathscr{A}}$ the numerical solution obtained via optimization procedures based on the loss function (\ref{eq3.17}). We then have the following result:
\begin{theorem} \label{thm1}
	Under the assumption of Lemma~\ref{lemma1}, if we further assume that there exists an extension function $\widetilde{u}({\bm{x}},t, z({\bm{x}}, t))$ of the solution $u({\bm{x}}, t)$ satisfying $\widetilde{u} \in H^4(\widetilde{Q})$, then we have
	\begin{equation}
		\begin{aligned}
			\|u - u^*_{d \mathscr{A}}\|^2_{L^2(0,T, H^1(\Omega))} \lesssim & \mathop{\inf}_{v_\theta \in \mathcal{H}} \|v_\theta - \widetilde{u}\|_{H^4(\widetilde{Q})}^2 + \mathop{\sup}_{v_\theta \in \mathcal{H}} \left|  {\mathcal{L}}^d(v_\theta) - {\mathcal{L}}(v_\theta) \right| \\
			&+ \left|  {\mathcal{L}}^d(\widetilde{u}^*_{d\mathscr{A}}) -  {\mathcal{L}}^d(\widetilde{u}^*_{d}) \right|.
		\end{aligned}
	\end{equation}
\end{theorem}
\begin{proof}
	Let $e:=u - u^*_{d \mathscr{A}}$, where $u^*_{d \mathscr{A}}\left( {\bm{x}}, t\right) := \widetilde{u}^*_{d \mathscr{A}}\left( {\bm{x}}, t, z\left( {\bm{x}}, t\right) \right) ,  {\bm{x}} \in \Omega,  t \in [0,T_\text{end}].$
It is easy to see that $e$ satisfies
	\begin{equation}
		\left\{
		\begin{aligned}
			\partial_t e - \nabla \cdot \left( \beta \nabla e\right)  &= f - \partial_t u^*_{d \mathscr{A}} + \nabla \cdot \left( \beta \nabla u^*_{d \mathscr{A}}\right) \quad  &&\mathrm{in}\; \Omega, \quad  t \in [0, T_{\text{end}}],  \\
			e\left( \cdot, t\right) &= g - u^*_{d \mathscr{A}} \quad  &&\mathrm{on}\; \partial \Omega, \quad t \in [0, T_{\text{end}}], \\
			e\left( \cdot, 0\right) &= u_0 - u^*_{d \mathscr{A},0}\quad &&\mathrm{in}\; \Omega,
		\end{aligned}\right.
	\end{equation}
	with the interface condition
	\begin{equation}
		\begin{aligned}
			[e]_{\Gamma(t)} &= 0, \quad t\in [0,T_{\text{end}}],\\
			[\beta \nabla e \cdot \mathbf{n}]_{\Gamma(t)}& = h_N - [\beta \nabla u^*_{d \mathscr{A}}  \cdot \mathbf{n}]_{\Gamma(t)}, \quad t\in [0,T_{\text{end}}]. 			
		\end{aligned}
	\end{equation}
Thus, from Lemma~\ref{lemma1} and the fact $\mathcal{L}(\widetilde{u}) = 0$, it follows that:
	\begin{align*}
		\|e\|_{L^2(0,T, H^1(\Omega))}^2 
		\lesssim & \| \partial_t u^*_{d \mathscr{A}} - \nabla \cdot \left( \beta \nabla u^*_{d \mathscr{A}} \right) - f \|_{L^2(0,T_{\text{end}};L^2(\Omega))}^2  \\
		&+ \|u^*_{d \mathscr{A}} - g\|_{L^2\left( 0, T_{\text{end}}, H^2(\partial \Omega)\right) }^2 + \| u^{*'}_{d \mathscr{A}} - g' \|_{L^2\left( 0, T_{\text{end}}, L^2(\partial \Omega)\right)}^2\\
		&+\|u^*_{d \mathscr{A}, 0} - u_0\|_{H^1(\Omega)}^2  + \|[\beta \nabla u^*_{d \mathscr{A}} \cdot \mathbf{n}]_{\Gamma(t)} - h_N\|_{L^2\left( 0, T_{\text{end}}; L^2(\Gamma(t))\right) }^2\\
		=& \mathcal{L}(\widetilde{u}^*_{d \mathscr{A}})={\mathcal{L}}(\widetilde{u}^*_{d \mathscr{A}}) -  {\mathcal{L}}\left( \widetilde{u}\right).
	\end{align*}
Further, it follows from the definitions of $\widetilde{u}^*$ and $\widetilde{u}^*_d$, and the fact $\mathcal{L}(\widetilde{u}) = 0$ again that
	\begin{equation*}
		\begin{aligned}
			 {\mathcal{L}}(\widetilde{u}^*_{d \mathscr{A}})& -  {\mathcal{L}}\left( \widetilde{u}\right)
			= \left(  {\mathcal{L}}(\widetilde{u}^*_{d \mathscr{A}}) -  {\mathcal{L}}^d(\widetilde{u}^*_{d \mathscr{A}})\right) + \left(  {\mathcal{L}}^d(\widetilde{u}^*_{d \mathscr{A}}) -  {\mathcal{L}}^d(\widetilde{u}^*_{d})\right)  \\
			&\qquad\qquad+ \left(  {\mathcal{L}}^d(\widetilde{u}^*_{d}) -  {\mathcal{L}}^d(\widetilde{u}^*) \right)+ \left(  {\mathcal{L}}^d(\widetilde{u}^*)  -  {\mathcal{L}}(\widetilde{u}^*) \right) + \left(  {\mathcal{L}}(\widetilde{u}^*) - {\mathcal{L}}\left( \widetilde{u}\right)\right) \\
			&\le \left(  {\mathcal{L}}(\widetilde{u}^*) -  {\mathcal{L}}\left( \widetilde{u}\right)\right) + \left(  {\mathcal{L}}^d(\widetilde{u}^*)  -  {\mathcal{L}}(\widetilde{u}^*) \right) + \left(  {\mathcal{L}}(\widetilde{u}^*_{d \mathscr{A}}) -  {\mathcal{L}}^d(\widetilde{u}^*_{d \mathscr{A}})\right)\\&\quad + \left(  {\mathcal{L}}^d(\widetilde{u}^*_{d \mathscr{A}}) -  {\mathcal{L}}^d(\widetilde{u}^*_{d})\right) \\
			&\le \mathop{\inf}_{v_\theta \in \mathcal{H}} \left|  {\mathcal{L}}(v_\theta) -  {\mathcal{L}}(\widetilde{u}) \right|  + 2\mathop{\sup}_{v_\theta \in \mathcal{H}} \left|  {\mathcal{L}}^d(v_\theta) -  {\mathcal{L}}(v_\theta) \right| + \left|  {\mathcal{L}}^d(\widetilde{u}^*_{d\mathscr{A}}) -  {\mathcal{L}}^d(\widetilde{u}^*_{d}) \right|.
		\end{aligned}
	\end{equation*}
Then, it suffices to bound the term $ {\mathcal{L}}(v_\theta) -  {\mathcal{L}}(\widetilde{u})$. From the definition of $ {\mathcal{L}}$ and Trace Theorem, it follows that
	\begin{equation*}
		\begin{aligned}
			 {\mathcal{L}}(v_\theta) -  {\mathcal{L}}(\widetilde{u}) 
			&= \int_{0}^{T_{\text{end}}}\int_{\Omega} \left| \partial_t v_\theta - \nabla \cdot (\beta \nabla v_\theta) - \partial_t \widetilde{u} +  \nabla \cdot (\beta \nabla \widetilde{u})\right|^2 \mathrm{d}{\bm{x}} \mathrm{d}t \\
			&\quad+\int_{0}^{T_{\text{end}}}\int_{\partial \Omega} { \mathop{\sum}_{|\alpha|=0}^2}\left| \partial^\alpha_{\bm{x}} \left( v_\theta - \widetilde{u}\right)   \right|^2 \mathrm{d}{\bm{x}} \mathrm{d}t +\int_{0}^{T_{\text{end}}}\int_{\partial \Omega}  \left|\partial_t v_\theta - \partial_t \widetilde{u} \right|^2 \mathrm{d}{\bm{x}} \mathrm{d}t \\	
			&\quad+\int_{\Omega} {\mathop{\sum}_{|\alpha| = 0}^1} \left| \partial^\alpha_{\bm{x}} \left(v_\theta({\bm{x}}, 0, z({\bm{x}}, 0)) - \widetilde{u}({\bm{x}}, 0)\right)  \right|^2 \mathrm{d}{\bm{x}} \\
			&\quad+\int_{0}^{T_{\text{end}}}\int_{\Gamma(t)} \left| [\beta \nabla_{\bm{x}} v_\theta \cdot \mathbf{n}]_{\Gamma(t)} - [\beta \nabla_{\bm{x}} \widetilde{u} \cdot \mathbf{n}]_{\Gamma(t)})\right|^2 \mathrm{d}{\bm{x}} \mathrm{d}t\\
			&\lesssim \|v_\theta - \widetilde{u}\|_{H^4(\widetilde{Q})}^2.
		\end{aligned}
	\end{equation*}
This completes the proof immediately.
\end{proof}

Following Theorem~\ref{thm1}, we decompose the error into three distinct components:  
\begin{itemize}
	\item \textbf{Approximation error $\mathcal{E}_\text{approx}$:} Best approximation error of the target function $\widetilde{u} \in H^4(\widetilde{Q})$ in the neural network function class $\mathcal{H}$ under the Sobolev norm on bounded domain $\widetilde{Q}$;
	\item \textbf{Statistical error $\mathcal{E}_\text{stat}$:} Error induced by finite sampling;
	\item \textbf{Optimization error $\mathcal{E}_\text{opt}$:} Error from numerical optimization procedures.
\end{itemize}

Under over-parameterized assumptions, the optimization error can be assumed arbitrarily small, as established in the literature~\cite{liu2022loss}. {Therefore, we assume that the optimization error $\mathcal{E}_{\text{opt}}$ can be  arbitrarily small. For approximation error $\mathcal{E}_\text{approx}$ in Sobolev space,} the following approximation theorem holds {\cite[Proposition 4.8]{guhring2021approximation}:}

\begin{theorem} \label{thm2}
	Let $s\in \mathbb{N} \cup \left\lbrace 0\right\rbrace $ be fixed, and $u \in H^k\left( D \right) $ with $k \ge s+1$, where $D$ is a bounded domain in $\mathbb{R}^{d}$. Then for any tolerance $\varepsilon > 0$, there exists at least one $v_\theta$ of depth $\mathcal{O}\left(\log\left( d+k\right) \right) $ with $|\theta|_{l^0}$ bounded by $\mathcal{O}\left( \varepsilon^{-\frac{d}{k-s-\mu(s=2)}}\right) $ and $|\theta|_{l^\infty}$ by $\mathcal{O}\left( \varepsilon^{-2 -\frac{9d/2+2s+2\mu(s=2)}{k-s-\mu(s=2)}}\right)  $, where $\mu$ is arbitrarily small, and the notation $(s=2)$ equals to $1$ if $s=2$ and $0$ otherwise,  such that
	\begin{equation*}
		\|v_\theta - u\|_{H^k(D)} \le \varepsilon.
	\end{equation*}
\end{theorem}
Theorem~\ref{thm2} demonstrates that the approximation error $\mathcal{E}_{\text{approx}}$ of the neural network can be effectively controlled. Therefore, for $\varepsilon>0$, there exists a DNN $v_\theta \in \mathcal{H}$ such that $\|v_\theta - \widetilde{u}\|_{H^4(\widetilde{Q})} \le \varepsilon$.

Next, we bound the statistical error via the Rademacher complexity $\mathfrak{R}_n$ which measures the complexity of a collection of function by the correlation between function values with Rademacher random variables, i.e., with probability $P\left( \omega =1\right) = P\left( \omega = -1\right) = \frac{1}{2}$ (cf.~\cite{goar2024foundations}). 

To derive an upper bound estimate for the statistical error $\mathcal{E}_\text{stat}$, we invoke the following several lemmas:
\begin{definition}
	Let $(\mathcal{M}, m)$ be a metric space of real-valued functions, and $\mathcal{G} \subset \mathcal{M}$. A set $\left\lbrace x_i\right\rbrace_{i=1}^p \subset \mathcal{G}$ is called an $\epsilon$-cover of $\mathcal{G}$ if for any $x\in\mathcal{G}$, there exists an element $x' \in \left\lbrace x_i\right\rbrace_{i=1}^p$ such that $m(x, x') \le \epsilon$. The $\epsilon$-covering number, denoted by $\mathcal{C}(\mathcal{G},m ,\epsilon)$, is the minimum cardinality among all $\epsilon$-covers of $\mathcal{G}$. That is,
	\begin{equation*}
		\mathcal{C}(\mathcal{G},m ,\epsilon) = \min \left\lbrace p\in \mathbb{N}: \exists \text{ an } \epsilon \text{-cover } \left\lbrace x_1,...,x_p\right\rbrace \text{ of } \mathcal{G}\right\rbrace. 
	\end{equation*}
\end{definition}
\begin{lemma}[Dudley's Lemma]\cite[Lemma A.6.]{hu2024solving}\label{lemma5}
	Let $\mathcal{F}$ be a real-valued functions mapping from $D$ to $\mathbb{R}$ class and $M_\mathcal{F}:=\mathop{\sup}_{f\in\mathcal{F}} \|f\|_{L^\infty(D)}$. Then we have
	\begin{equation*}
		\mathfrak{R}_n(\mathcal{F}) \le \mathop{\inf}_{0<s<M_\mathcal{F}} \left( 4s + 12 n^{-1/2} \int_{s}^{M_\mathcal{F}} \left( \log \mathcal{C}\left( \mathcal{F}, \|\cdot\|_{L^\infty(D)}, \epsilon\right) \right) ^{1/2} \mathrm{d}\epsilon\right) .
	\end{equation*}
\end{lemma}
\begin{theorem}
	Let $L$, $W$ and $B_\theta$ be the depth, width, and maximum weight bound of a DNN function class $\mathcal{H}$ with $N_\theta$ nonzero weights. We assume the assumptions in Lemma~\ref{lemma1} hold and the level set function $\phi({\bm{x}}, t) \in C^2(Q)$. {Then,} for any small $0<\delta < 1$, with probability at least $\left( 1 - \delta\right)^4$, {it} holds
	\small{
	\begin{equation*}
		\mathop{\sup}_{v_\theta \in \mathcal{H}} \left|  {\mathcal{L}}^d(v_\theta) -  {\mathcal{L}}(v_\theta) \right| 
		\lesssim  \dfrac{\log^{\frac{1}{2}} N_{\Omega} + 1}{\sqrt{N_{\Omega}}} + \dfrac{\log^{\frac{1}{2}} N_{\partial \Omega} + 1}{\sqrt{N_{\partial \Omega}}} 
		+ \dfrac{\log^{\frac{1}{2}} N_{\Omega_0} + 1}{\sqrt{N_{\Omega_0}}} + \dfrac{\log^{\frac{1}{2}} N_{\Gamma}+ 1}{\sqrt{N_{\Gamma}}},
	\end{equation*}
	}
	where $N_{\Omega}$, $N_{\partial \Omega}$, $N_{\Omega_0}$ and $ N_{\Gamma}$ denote the numbers of uniform sample points in their corresponding domain.
\end{theorem}
\begin{proof}
It is obvious that
	\begin{equation*}
		\mathop{\sup}_{v_\theta \in \mathcal{H}} \left|  {\mathcal{L}}^d(v_\theta) -  {\mathcal{L}}(v_\theta) \right| \le \mathop{\sup}_{v_\theta \in \mathcal{H}} \mathop{\sum}_{i=1}^4 \left|  {\mathcal{L}}_i(v_\theta) - \widehat{\ell}_i(v_\theta) \right|.
	\end{equation*}
For the first term $\mathop{\sup}_{v_\theta \in \mathcal{H}}\left|  {\mathcal{L}}_1(v_\theta) - \widehat{\ell}_1(v_\theta) \right|$, by use of {\cite[Theorem 3.3]{goar2024foundations}}, we can obtain the following inequality with probability at least $(1-\delta)$
	\begin{equation*}
		\mathop{\sup}_{v_\theta \in \mathcal{H}}\left|  {\mathcal{L}}_1(v_\theta) - \widehat{\ell}_1(v_\theta) \right|
		\lesssim \mathfrak{R}_{N_{\Omega}}\left( \psi \circ \left( \mathcal{D}\left( \mathcal{H}\right) - f\right) \right) + N_{\Omega}^{-1/2},
	\end{equation*}
where $\psi(x) = x^2$ and $
		\mathcal{D}\left( \mathcal{H}\right) := \left\lbrace \partial_t v_\theta - \nabla \cdot (\beta \nabla v_\theta) | v_\theta \in \mathcal{H} \right\rbrace. $
	
{Combining the assumptions on the neural network function class $\mathcal{H}$ with \cite[Lemma 2]{siegel2023greedy}, and using the L-Lipschitz continuity of $\psi$, we directly obtain:}
	\begin{equation} \label{eq3.34}
		\mathfrak{R}_{N_{\Omega}}\left( \psi \circ \left( \mathcal{D}\left( \mathcal{H}\right) - f\right) \right) \lesssim \mathfrak{R}_{N_{\Omega}}\left( \mathcal{D} \left( \mathcal{H}\right)\right) + \mathfrak{R}_{N_{\Omega}}(f).
	\end{equation}
	
	By Remark~\ref{remark2}, for $v_\theta \in \mathcal{H}$ we can explicitly write the following equation
	\begin{equation*}
		\partial_t v_\theta - \nabla \cdot(\beta \nabla v_\theta) = D_t v_\theta + D_z v_\theta \partial_t z - \beta \left( \Delta_{\bm{x}} + 2 \nabla z \cdot \nabla_{\bm{x}} (D_z v_\theta) + |\nabla z|^2 D_z^2 + D_z v_\theta \Delta z\right). 
	\end{equation*}
	
	Since $\phi({\bm{x}}, t) \in C^2(Q)$ and $Q$ is a bounded domain, we know from {\cite[Section 5]{jin2023solving}} that $\mathcal{D}(\mathcal{H})$ constitutes a class of bounded functions which are Lipschitz continuous for parameter $\theta$. We denote the Lipschitz constant by $L_\mathcal{\mathcal{D}(\mathcal{H})}$ and the upper bound of $\mathcal{D}(\mathcal{H})$ by $M_{\mathcal{D}(\mathcal{H})}$, both of {which} depend on $W$, $B_\theta$, $L$ and $N_\theta$. Fixed $B_\theta \in [1, \infty)$ and $\epsilon \in (0,1)$, applying the result of~\cite[Proposition 5]{cucker2001on} yields
	\begin{equation*}
		\log \mathcal{C} \left( \mathcal{D}(\mathcal{H}), \|\cdot\|_{L^\infty(Q)}, \epsilon \right) \le \log \mathcal{C} \left( \theta, |\cdot|_{l^\infty}, L_{\mathcal{D}(\mathcal{H})}^{-1} \epsilon\right) \le N_\theta \log( 4 B_\theta L_{\mathcal{D}(\mathcal{H})} \epsilon^{-1}).
	\end{equation*}
Further, setting $s = N_{\Omega}^{-1/2}$ in Lemma~\ref{lemma5}, we have
	\begin{equation*}
		\begin{aligned}
			\mathfrak{R}_{N_{\Omega}}(\mathcal{D}(\mathcal{H}))
			&\le 4 N_{\Omega}^{-1/2} + 12 N_{\Omega}^{-1/2} \int_{N_{\Omega}^{-1/2}}^{M_{\mathcal{D}(\mathcal{H})}} \left( N_\theta \log\left(4 B_\theta L_{\mathcal{D}(\mathcal{H})} \epsilon^{-1}\right) \right) ^{1/2} \mathrm{d}\epsilon \\
			&\le 4 N_{\Omega}^{-1/2} + 12 N_{\Omega}^{-1/2} \left( M_{\mathcal{D}(\mathcal{H})} - N_{\Omega}^{-1/2} \right) \left( N_\theta \log\left(4 B_\theta L_{\mathcal{D}(\mathcal{H})} N_{\Omega}^{1/2}\right) \right) ^{1/2} \\
			&\le 4 N_{\Omega}^{-1/2} + 12 M_{\mathcal{D}(\mathcal{H})}N_\theta^{1/2} N_{\Omega}^{-1/2}\left( \log^{1/2} (4 B_\theta L_{\mathcal{D}(\mathcal{H})}) + \log^{1/2} \left( N_{\Omega}^{1/2}\right) \right) \\
			&\lesssim N_{\Omega}^{-1/2} \left( \log^{1/2} N_{\Omega} + 1\right). 
		\end{aligned}
	\end{equation*}	
	To bound the second term in (\ref{eq3.34}), we apply Jensen's inequality and the fact that $\mathbb{E}[\omega_i \omega_j] = \delta_{ij}$, which yields
	\begin{equation*}
		\begin{aligned}
			\mathfrak{R}_{N_{\Omega}}\left( f\right) 
			&= \mathbb{E}_{\xi,\omega} \left[ N_{\Omega}^{-1} \left| \mathop{\sum}_{j=1}^{N_{\Omega}} \omega_j f(\xi_j)\right| \right]\le \mathbb{E}_\xi \left[ \mathbb{E}_\omega \left( N_{\Omega}^{-1} \mathop{\sum}_{j=1}^{N_{\Omega}} \omega_j f(\xi_j)\right) ^2\right] ^{1/2} \\
			&= N_{\Omega}^{-1} \mathbb{E}_\xi \left[ \mathop{\sum}_{j=1}^{N_{\Omega}} f^2(\xi_j)\right] ^{1/2} = N_{\Omega}^{-1/2} \left( \int_0^{T_{\text{end}}} \int_{\Omega(t)} f^2({\bm{x}}, t) \mu({\bm{x}}, t) \mathrm{d}{\bm{x}}\mathrm{d}t\right)^{1/2} \\
			&= \left|Q\right|^{-1/2} N_{\Omega}^{-1/2}\|f\|_{L^2(0,T_{\text{end}}; L^2(\Omega))} \lesssim N_{\Omega}^{-1/2},
		\end{aligned}
	\end{equation*}
	where $\left\lbrace \xi_j\right\rbrace_{j=1}^{N_{\Omega}}$ are i.i.d. samples from the uniform distribution $\mathcal{U}\left( Q\right) $, $\mu(\bm{x}, t)$ represents the corresponding probability density function on $Q$ and $|Q|$ denotes the measure of the domain $Q$. 
    Then we obtain that 
	\begin{equation*}
		\mathop{\sup}_{v_\theta \in \mathcal{H}}\left|  {\mathcal{L}_1(v_\theta) - \widehat{\ell}_1(v_\theta)} \right| \lesssim \dfrac{\log^{\frac{1}{2}} N_{\Omega} + 1}{\sqrt{N_{\Omega}}}.
	\end{equation*}
The proof is completed by analogous arguments for the other terms.
\end{proof}

As established in Section~\ref{sec2.4}, we have presented a method for approximating the level set function using neural networks, which is denoted by $\phi_\theta({\bm{x}}, t)$. The zero level set of $\phi_\theta({\bm{x}}, t)$ represents the approximate interface
\begin{equation*}
	\Gamma_\theta(t) :=\left\{ {\bm{x}} | \phi_\theta({\bm{x}}, t) = 0, {\bm{x}} \in \Omega\right\}, \quad t\in[0, T_{\text{end}}].
\end{equation*}
Although \(\Gamma_\theta(t)\) deviates from the true interface \(\Gamma(t)\), we establish that the solution remains controllable relative to the original problem's solution under suitable assumptions with sufficiently small interface discrepancy. We quantify this interface deviation via the Hausdorff distance:
$\delta = \mathop{\max}_{t \in [0, T_{\text{end}}]} d_H\left( \Gamma(t), \Gamma_\theta(t)\right),$
where $d_H$ denotes Hausdorff distance:
\begin{equation*}
	d_H(A, B) := \mathop{\max}\left\lbrace \mathop{\sup}_{a\in A} \mathop{\inf}_{b\in B} \mathrm{dist}\left( a, b\right) , \mathop{\sup}_{b\in B} \mathop{\inf}_{a\in A} \mathrm{dist}\left( a, b\right) \right\rbrace. 
\end{equation*}

We define a tubular neighborhood of width $\delta > 0$  about the interface $\Gamma(t)$ as follows:
\begin{equation*}
	S_{\delta}\left( t\right) :=\left\lbrace {\bm{x}}| \mathrm{dist}({\bm{x}}, \Gamma(t)) < \delta \right\rbrace,
	\quad S_\delta:= \mathop{\bigcup}_{t \in [0, T_{\text{end}}]}  S_\delta(t).
\end{equation*}

By use of the Lemma 2.1 in~\cite{li2010optimal}, we have:
\begin{lemma} \label{lemma6}
	For any $v \in B_{2,1}^{1/2}\left( \Omega^+(t)\cup \Omega^-(t)\right), t\in[0,T_{\text{end}}]$, it holds
	\begin{equation*}
		\|v\|_{L^2(S_\delta(t))}^2 \lesssim \delta \|v\|_{B_{2,1}^{1/2}\left( \Omega^+(t)\cup \Omega^-(t)\right)}.
	\end{equation*}
\end{lemma}

Finally, we conclude the following result:
\begin{theorem}\label{thm3.8}
	Assume that the conditions of Eqs.(\ref{eq1.1})-(\ref{eq1.7}) satisfy {$h_D=0$ and $h_N=0$}. Then, we have the following estimate
	\begin{equation*}
		\|u-u_\delta\|_{L^2(0,T_{\text{end}}; H^1(\Omega))} \lesssim \delta^{1/2},
	\end{equation*}
	where $u$ is the solution to the equations with the true interface $\Gamma(t)$, and $u_\delta$ is the solution corresponding to the approximate interface $\Gamma_\theta(t)$.
\end{theorem}
\begin{proof}
	According to the assumptions, we have
	\begin{equation*}
		\int_\Omega \partial_t u(t) w(t) + \int_\Omega \beta \nabla u(t) \cdot \nabla w(t) = \int_\Omega f(t) w(t) ,\quad w(t) \in H^1_0(\Omega), t\in[0,T_{\text{end}}],
	\end{equation*}
	\begin{equation*}
		\int_\Omega \partial_t u_\delta(t) w(t) + \int_\Omega \beta_\delta \nabla u_\delta(t) \cdot \nabla w(t) = \int_\Omega f(t) w(t) ,\quad w(t) \in H^1_0(\Omega), t\in[0,T_{\text{end}}],
	\end{equation*}
	where
	\begin{equation*}
		\beta_\delta({\bm{x}}, t) = \left\lbrace 
		\begin{aligned}
			&\beta^+, \quad \mathrm{for}\;{\bm{x}} \in \Omega_\delta^+(t), \\
			&\beta^-, \quad \mathrm{for}\;{\bm{x}} \in \Omega_\delta^-(t), 
		\end{aligned}
		\right. 
	\end{equation*}
	and
	\begin{equation*}
		\Omega_\delta^+(t) = \left\lbrace {\bm{x}}| \phi_\theta({\bm{x}}, t)>0, {\bm{x}} \in \Omega \right\rbrace, \quad \Omega_\delta^-(t) = \left\lbrace {\bm{x}}| \phi_\theta({\bm{x}}, t)<0, {\bm{x}} \in \Omega \right\rbrace, \quad  t\in [0,T_{\text{end}}].
	\end{equation*}
	It follows that
	\begin{equation*}
		\int_\Omega \partial_t\left( u_\delta - u\right) w + \int_\Omega \beta_\delta \nabla\left( u_\delta - u\right) \cdot \nabla w = \int_{S_{\delta}(t)} \left( \beta - \beta_\delta\right) \nabla u \cdot \nabla w.
	\end{equation*} 
Further, letting $w = u_\delta - u$ and applying the Cauchy-Schwarz and Young's inequalities, it follows from Lemma~\ref{lemma6} that
	\begin{equation*}
		\dfrac{\mathrm{d}}{\mathrm{d}t} \|\left( u_\delta - u\right)(t)\|_{L^2(\Omega)}^2 + \|\nabla\left( u_\delta - u \right) (t)\|_{L^2(\Omega)}^2 \lesssim \|\nabla u\|_{L^2(S_\delta(t))}^2.
	\end{equation*}
	Then, adding $\| \left( u_\delta - u\right)(t)\|_{L^2(\Omega)}$ to both sides of the inequality yields
	\begin{equation} \label{eq3.44}
		\begin{split}
		\dfrac{\mathrm{d}}{\mathrm{d}t} \|\left( u_\delta - u\right)(t)\|_{L^2(\Omega)}^2 &+ \|\left( u_\delta - u \right) (t)\|_{H^1(\Omega)}^2\\ & \lesssim \|(u_\delta - u)(t)\|_{L^2(\Omega)}^2 + \|\nabla u\|_{L^2(S_\delta(t))}^2.
		\end{split}
	\end{equation}
Thus the differential form of Gronwall's inequality yields the estimate
	\begin{equation*}
		\mathop{\max}_{t \in [0,T_{\text{end}}]} \|\left( u_\delta - u\right)(t)\|_{L^2(\Omega)}^2 \lesssim \|(u_\delta - u)(0)\|_{L^2(\Omega)}^2 + \int_0^{T_{\text{end}}} \|\nabla u\|_{L^2(S_\delta(t))}^2.
	\end{equation*}
	Since the interface at time $t=0$ is known exactly, we integrate Equation (\ref{eq3.44}) with respect to time to obtain:
	\begin{equation}
		\begin{aligned}
			&\mathop{\max}_{t \in [0,T_{\text{end}}]} \|\left( u_\delta - u\right)(t)\|_{L^2(\Omega)}^2 + \int_{0}^{T_{\text{end}}}\|\left( u_\delta - u \right) (t)\|_{H^1(\Omega)}^2 \\
			&\qquad\qquad\quad\lesssim \int_0^{T_{\text{end}}}\|\nabla u\|_{L^2(S_\delta(t))}^2 
			\lesssim \int_0^{T_{\text{end}}}\|\nabla u\|_{B_{2,1}^{1/2}(\Omega^+(t) \cup \Omega^-(t))} \lesssim \delta.
		\end{aligned}
	\end{equation}
This completes the proof of the theorem.
\end{proof}

\section{Numerical experiments} \label{sec4}
In this section, we present several numerical examples to demonstrate the effectiveness and accuracy of the XI-PINN method for moving interface problems, particularly emphasizing its ability to capture solution discontinuities across the interface. In all test cases, we use the hyperbolic tangent activation function $\sigma(x) = \tanh(x)$. To minimize the loss function, we adopt the improved Levenberg--Marquardt (LM) method~\cite{transtrum2012improvements}. {Compared to commonly used first-order optimizers (e.g., gradient descent), second-order methods have been shown to effectively mitigate the ill-conditioned loss landscapes of PINNs~\cite{rathore2024challenges} and alleviate the gradient conflicts arising in multi-objective optimization~\cite{wang2026gradient}. Furthermore, while training PINNs via gradient descent often requires adaptively adjusting the penalty parameters of different residual terms to ensure convergence~\cite{wang2022and, wang2021understanding}.} {Since we adopt a second-order optimizer throughout this work, we simply fix all penalty parameters to \(1\) for all neural network methods, exactly as in the loss function defined in \eqref{eq2.7}. }  

In all the following experiments, the optimization is terminated after \(2000\) epochs. 
We compare the proposed XI-PINN with the Vanilla-PINN~\cite{raissi2019physics} and with X-PINNs~\cite{jagtap2020extended}.  
Vanilla-PINN, like XI-PINN, employs a single neural network to represent the solution in both subdomains, while X-PINNs use two separate networks to approximate the solution in \(\Omega^+(t)\) and \(\Omega^-(t)\), respectively.  
For Vanilla-PINN, the interface jump condition is approximated by a finite difference:
\begin{equation}
	\bigl[ \beta \nabla u \cdot \mathbf{n} \bigr]_{\Gamma(t)}
	\approx \beta^+ \nabla u(\bm{x} + \varepsilon \mathbf{n}, t) \cdot \mathbf{n}
	- \beta^- \nabla u(\bm{x}, t) \cdot \mathbf{n},
\end{equation}
where we take \(\varepsilon = 10^{-3}\).  
For X-PINNs, one only needs to enforce the interface conditions between the two networks (see, e.g., the treatment of elliptic interface problems in~\cite{wu2022inn}).  
When a neural network is used to construct the level set function \(\phi_\theta\) within the XI-PINN framework, we denote the resulting scheme as \(\text{XI-PINN}^*\) and set the threshold \(\delta = 0.2\) in Algorithm~\ref{algorithm 1}. {In the present work, we consider the case where the interface velocity field $\mathcal{V}$ is known; therefore, the level set function $\phi_\theta$ can be constructed independently prior to the training of XI-PINN.}
For a neural network approximation \(u_\theta\), we define the relative errors in the \(L^2(0,T;L^2(\Omega))\) and \(L^2(0,T;H^1(\Omega))\) norms as
\begin{equation*}
	E_{L^2}^{\text{Rel}}(u_\theta)
	= \frac{\|u - u_\theta\|_{L^2(0,T;L^2(\Omega))}}
	{\|u\|_{L^2(0,T;L^2(\Omega))}},
	\qquad
	E_{H^1}^{\text{Rel}}(u_\theta)
	= \frac{\|u - u_\theta\|_{L^2(0,T;H^1(\Omega))}}
	{\|u\|_{L^2(0,T;H^1(\Omega))}}.
\end{equation*}
All experiments were run on an NVIDIA V100-SXM2-32GB GPU.
\subsection{Parabolic equations}
\begin{example}\label{exa1}
	Consider the 2D moving interface problem (\ref{eq1.1})-(\ref{eq1.7}) defined over the spatial domain $\Omega=[-1,1]^2$ and the time interval $[0, 1]$. At time $t = 0$, the interface $\Gamma(0)$ is a circle centered at $(0.3, 0)$, represented by the level set function $\phi_0(x, y) = (x-0.3)^2 + y^2 - (\pi/6)^2$. Driven by a velocity field $\mathcal{V} = \left[ -0.3 \pi \sin(\pi t), 0.3 \pi \cos(\pi t)\right]^T $, the interface rotates about the origin $(0, 0)$. Thus, we can express the evolving level set function as:
	$$
	\phi(x, y, t) = (x - 0.3\cos(\pi t))^2 + (y - 0.3\sin(\pi t))^2 - (\pi/6)^2.
	$$
	The exact solution is as follows:
	\begin{equation*}
		\small
		u(x, y, t) = \left\lbrace 
		\begin{aligned}
			& \dfrac{\left( \left( x - 0.3 \cos(\pi t)\right)^2 + \left( y - 0.3 \sin(\pi t) \right)^2\right) ^{5/2} \left( \pi / 6\right) ^{-1}  }{\beta^+}\\
			&\qquad\qquad \qquad\qquad\qquad\qquad+ \left( \pi / 6\right)^4\left( \dfrac{1}{\beta^-} -  \dfrac{1}{\beta^+}\right) , \quad \mathrm{in}\; \Omega^+(t), \\
			& \dfrac{\left( \left( x - 0.3 \cos(\pi t)\right)^2 + \left( y - 0.3 \sin(\pi t) \right)^2\right) ^{5/2} \left( \pi / 6\right) ^{-1}  }{\beta^-}, \quad  \mathrm{in}\; \Omega^-(t),
		\end{aligned}
		\right. 
	\end{equation*}
	where {$\beta^+ = 10$ and $\beta^- = 0.1$}.
\end{example}

\begin{figure}[htbp] 
	\centering
	\includegraphics[width=0.9\linewidth]{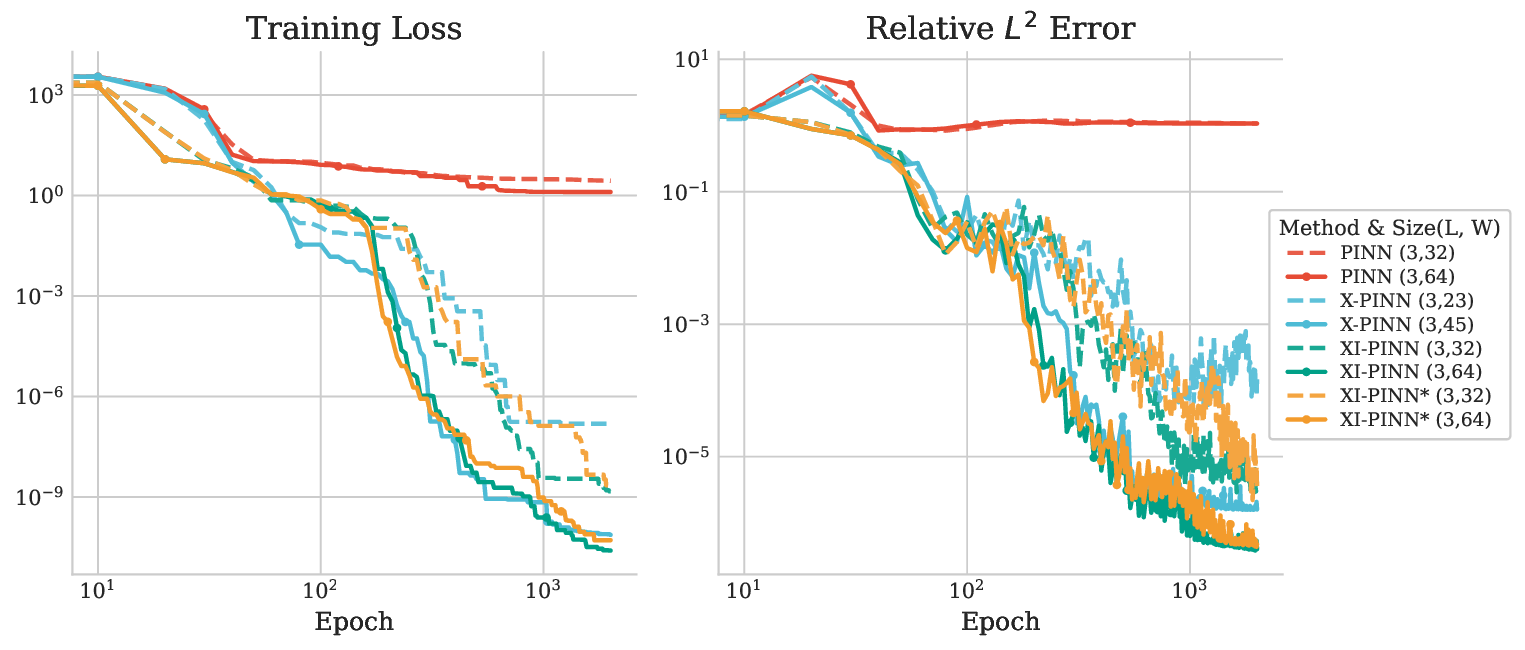}
	\caption{Example~\ref{exa1}: Decay of the loss and the relative $L^2(0,T;L^2(\Omega))$ error during training for different models with various network architectures.}
	\label{fig1}
\end{figure}
\begin{table}[htbp]
	\centering
	\smaller                    
	\setlength{\tabcolsep}{4pt}
	\caption{Performance of different numerical methods in Example~\ref{exa1}. (The network architecture $(L,W)$ refers to the solution network $u_\theta$.)}
	\label{tab1}
	\begin{tabularx}{\textwidth}{c|*{7}{>{\centering\arraybackslash}X}} 
		\toprule
		Method & $L$ & $W$ & $|\theta|_{\ell^0}$ & $\ell$ & $E^{\text{Rel}}_{L^2}$ & $E^{\text{Rel}}_{H^1}$ & Time (s) \\
		\midrule
		\multirow{2}{*}{Vanilla-PINN}
		& 3 & 32 & 1217  & 2.73e+0 & 1.09e+0 & 1.46e+0 & 137.0 \\
		& 3 & 64 & 4481  & 1.27e+0 & 1.09e+0 & 1.42e+0 & 201.5 \\
		\midrule
		\multirow{2}{*}{X-PINN}
		& 3 & 23 & 1336  & 1.54e-07 & 1.48e-04 & 2.14e-04 & 113.2 \\
		& 3 & 45 & 4592  & 7.42e-11 & 1.65e-06 & 1.03e-05 & 286.2 \\
		\midrule
		\multirow{2}{*}{XI-PINN}
		& 3 & 32 & 1249  & 1.49e-09 & 4.06e-06 & 7.46e-06 & 140.4 \\
		& 3 & 64 & 4545  & 2.52e-11 & 4.02e-07 & 1.50e-06 & 265.2 \\
		\midrule
		\multirow{2}{*}{XI-PINN*}
		& 3 & 32 & 1249  & 2.23e-09 & 4.42e-06 & 8.71e-06 & 203.3 \\
		& 3 & 64 & 4545  & 5.14e-11 & 4.81e-07 & 1.63e-06 & 319.6 \\
		\bottomrule
	\end{tabularx}
\end{table}

For this example we conduct a thorough comparison of Vanilla-PINN, X-PINNs, XI-PINN, and $\text{XI-PINN}^*$ using networks of different sizes.
In all the comparative experiments, {the numbers of training points are set to $(N_\Omega, N_{\partial\Omega}, N_{\Omega_0}, N_\Gamma)=(10\text{K}, 3\text{K}, 0.3\text{K}, 2\text{K})$, and we keep the total numbers of parameters of networks of different architectures as close as possible.}
{For $\text{XI-PINN}^*$, the level set function $\phi_\theta$ is constructed by a single fully connected network of architecture $(L,W)=(2,64)$, whose training takes $59.2$ seconds.}
The decay of the loss $\ell$ and the relative $L^2(0,T;L^2(\Omega))$ error during training are shown in Fig.~\ref{fig1}, and the final numerical results are reported in Table~\ref{tab1}.

The results show that Vanilla-PINN fails to yield an accurate solution because it cannot capture discontinuities at the interface. Although the computational costs (measured by training time and including the overhead for training $\widehat{\mathbf{X}}_\theta$ in $\text{XI-PINN}^*$) are similar across the three methods, XI-PINN achieves slightly higher accuracy than X-PINNs. {As shown by the error distributions in Fig.~\ref{fig2} and Fig.~\ref{fig3}, XI-PINN produces a smooth error field without interfacial jumps. This is because its input variables, $(\bm{x}, t, \Phi(\bm{x}, t))$, are constructed to be globally continuous. Furthermore, its single-network architecture shares parameters across subdomains, allowing for adaptive parameter allocation and leading to a more uniform error distribution. In contrast, X-PINNs use independent networks for each subdomain. This multi-network architecture only approximates the interface conditions and does not exactly satisfy continuity, which introduces localized error jumps. Moreover, due to the strongly coupled interface conditions, an error in one network easily propagates into the interior of the other subdomains. This error propagation becomes much more obvious when smaller networks are used (see Fig.~\ref{fig3}).}

\begin{figure}[htbp]
	\centering
	\begin{minipage}{1 \linewidth}
		\centering
		\includegraphics[width=0.9\linewidth]{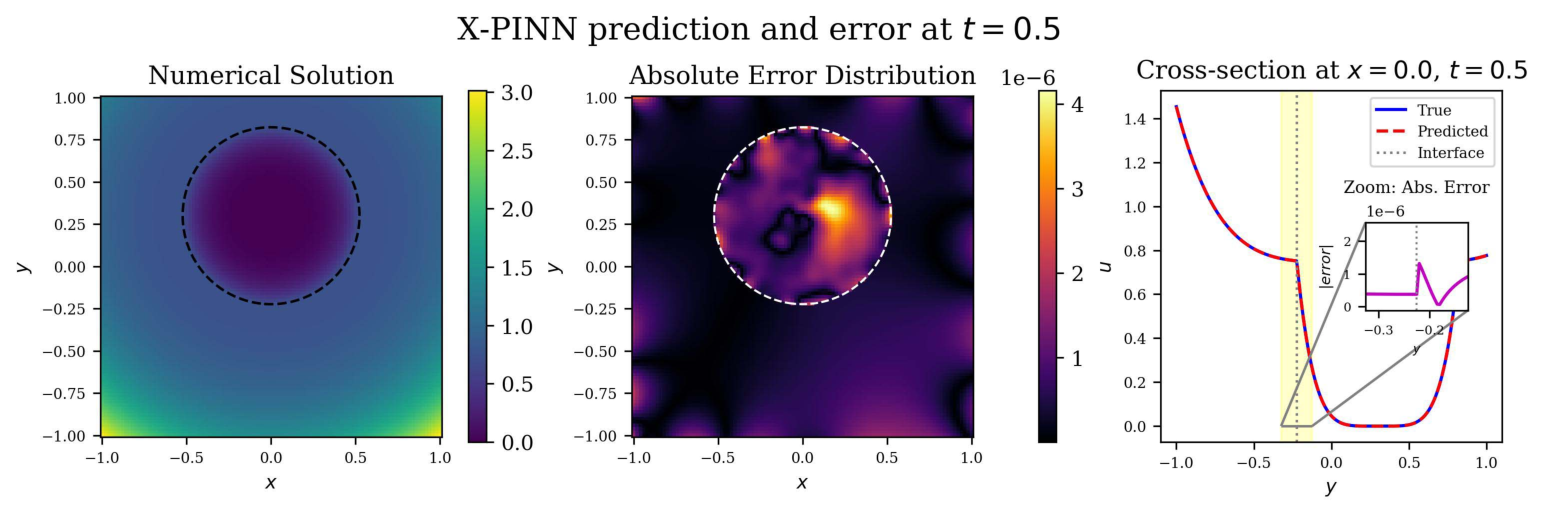}
	\end{minipage}
	\\
	\begin{minipage}{1 \linewidth}
		\centering
		\includegraphics[width=0.9\linewidth]{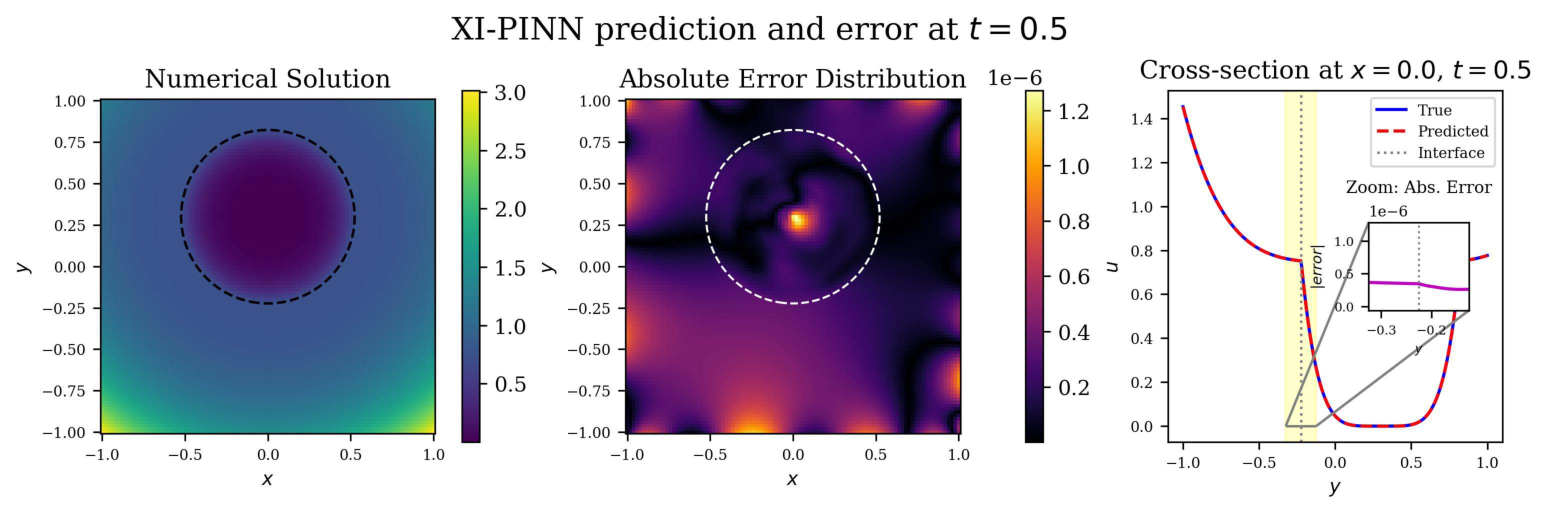}
	\end{minipage}
	\caption{Example~\ref{exa1}: Numerical solutions, absolute error distributions, and profiles along $x=0$ at $t=0.5$ for X-PINN and XI-PINN. The network sizes $(L,W)$ are $(3,45)$ and $(3,64)$, respectively.}
	\label{fig2}
\end{figure}

\begin{figure}[htbp]
	\centering
	\begin{minipage}{1 \linewidth}
		\centering
		\includegraphics[width=0.9\linewidth]{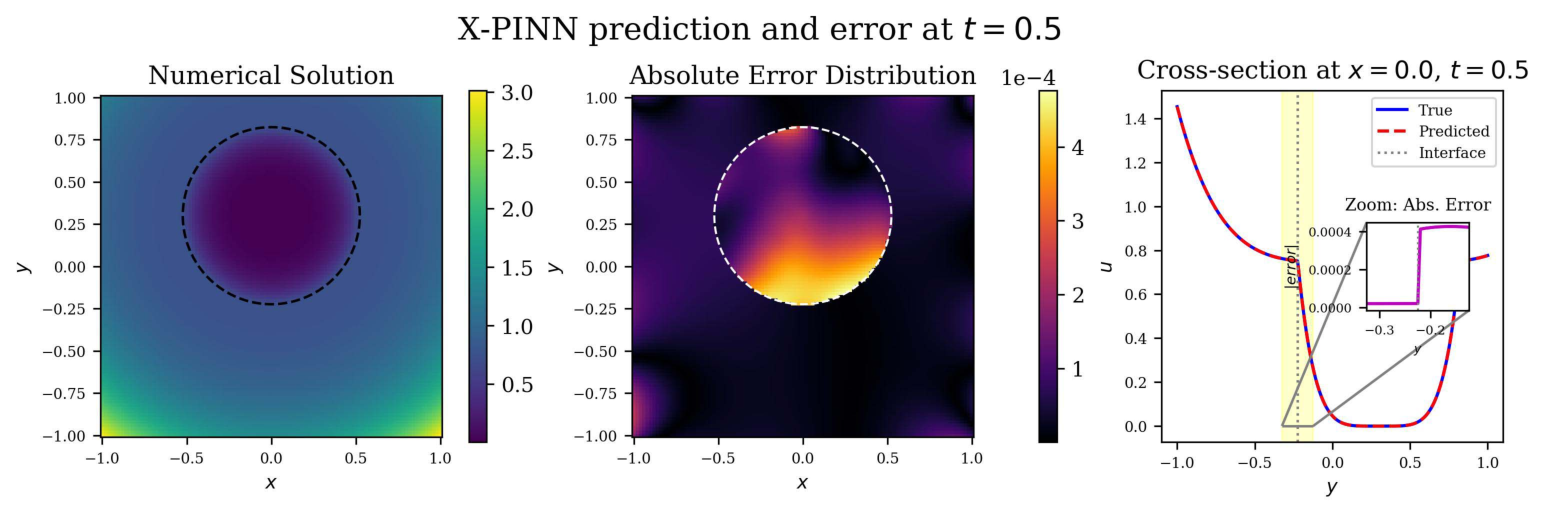}
	\end{minipage}
	\\
	\begin{minipage}{1 \linewidth}
		\centering
		\includegraphics[width=0.9\linewidth]{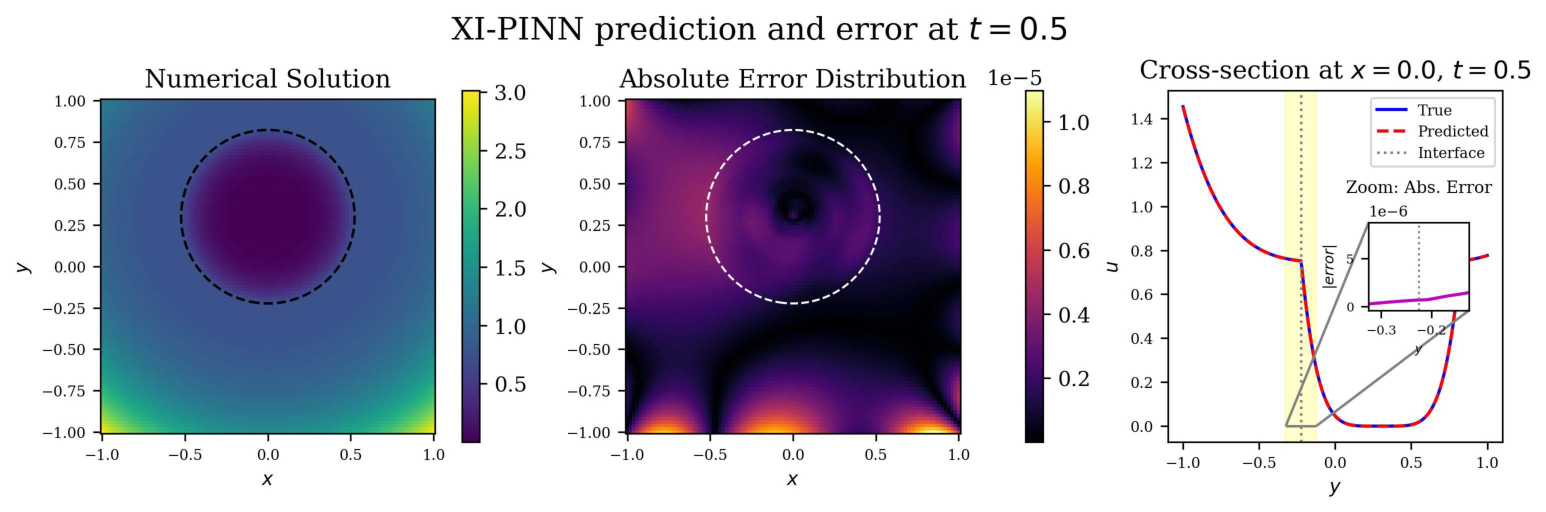}
	\end{minipage}
	\caption{Example~\ref{exa1}: Numerical solutions, absolute error distributions, and profiles along $x=0$ at $t=0.5$ for X-PINN and XI-PINN. The network sizes $(L,W)$ are $(3,23)$ and $(3,32)$, respectively.}
	\label{fig3}
\end{figure}

For $\text{XI-PINN}^*$, the zero level set of the learned function $\phi_\theta$ has a slight geometric deviation from the exact interface. This geometric error causes a small drop in numerical accuracy, which is consistent with the theoretical bounds established in Theorem~\ref{thm3.8}. {Finally, we note a trade-off in the optimization process. While parameter sharing in XI-PINN improves accuracy, it results in a dense Jacobian matrix in the Levenberg-Marquardt (LM) algorithm. Conversely, the independent networks in X-PINNs naturally yield a block-diagonal Jacobian matrix, which provides a more favorable structure for computation.}

\begin{table}[htbp]
	\caption{Example~\ref{exa1}: Relative $L^2(0,1; L^2(\Omega))$ and $L^2(0,1; H^1(\Omega))$ errors for different numbers of training points with network architecture $(L,W)=(3,64)$.}
	\label{tab2}
	\centering
	\small
	\begin{tabular}{c|ccc}
		\toprule
		$(N_\Omega, N_{\partial \Omega}, N_{\Omega_0}, N_{\Gamma})$ & $\ell$ & $E^{\text{Rel}}_{L^2}$ & $E^{\text{Rel}}_{H^1}$ \\
		\midrule
		(10K, 3K, 0.3K, 2K)    & 2.52e-11 & 4.02e-07 & 1.50e-06\\
		(5K, 2K, 0.2K, 1K)     & 1.18e-11 & 8.13e-07 & 4.02e-06\\
		(2K, 1K, 0.1K, 0.5K)   & 1.40e-12 & 1.51e-05 & 8.46e-05\\
		(1K, 0.5K, 0.1K, 0.3K) & 3.02e-13 & 2.86e-04 & 1.08e-03\\
		\bottomrule
	\end{tabular}
\end{table}

For the XI-PINN, we investigate the effect of statistical error on the numerical accuracy by varying the number of training points. 
The numerical results in Table~\ref{tab2} show that, as the number of training points decreases, the approximation error grows even though the loss $\ell$ remains very small. 
This indicates that reducing training points increases the statistical error, which eventually degrades the accuracy---a behavior consistent with the analysis in Section~\ref{sec3}.

Both Vanilla-PINN and XI-PINN are able to enforce continuity across the interface, but Vanilla-PINN fails to capture the normal flux discontinuity and the overall training fails. 
To understand this, we examine the NTK matrices of the two methods. 
For this comparison we use a single-hidden-layer network with $512$ neurons for both models, and set the numbers of training points to $(N_\Omega, N_{\partial\Omega}, N_{\Omega_0}, N_\Gamma)=(1000, 500, 200, 300)$. 
The eigenvalue spectra of the NTK matrices are plotted in Fig.~\ref{fig4}. 
We observe that the eigenvalues of the XI-PINN NTK decay more slowly than those of Vanilla-PINN. 
Since small eigenvalues of the NTK are typically associated with high-frequency components of the solution\cite{jacot2018neural, wang2022and}, XI-PINN is better able to capture the interface discontinuity while still maintaining continuity across the interface. 
We do not compare with X-PINNs here, because X-PINNs deal with a piecewise smooth problem, which is fundamentally different from the setting of Vanilla-PINN and XI-PINN.

\begin{figure}[htbp]
	\centering
	\begin{minipage}{0.25\linewidth}
		\centering
		\includegraphics[width=1\linewidth]{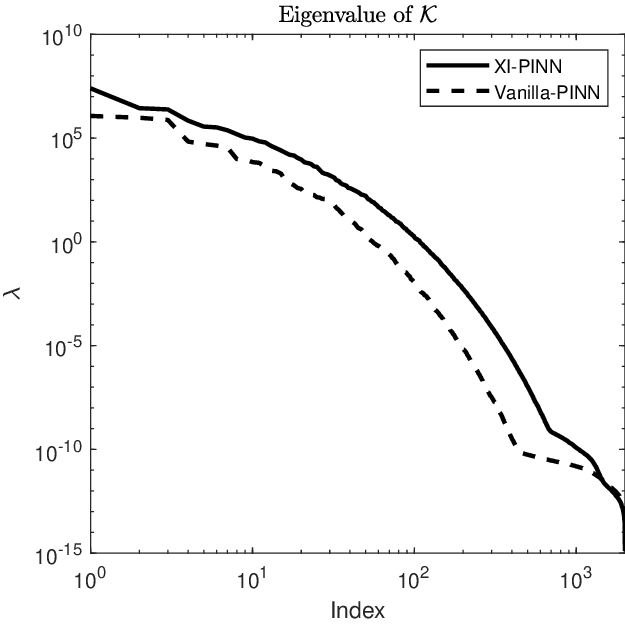}
	\end{minipage}
	\begin{minipage}{0.25\linewidth}
		\centering
		\includegraphics[width=1\linewidth]{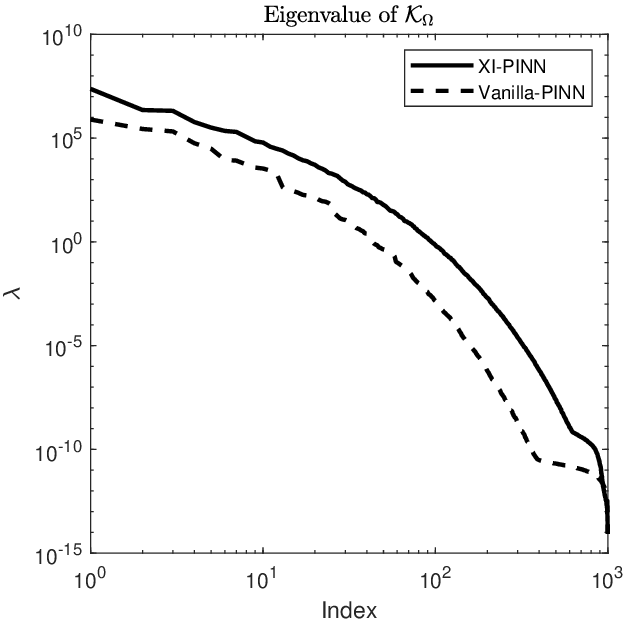}
	\end{minipage}
	\begin{minipage}{0.25\linewidth}
		\centering
		\includegraphics[width=1\linewidth]{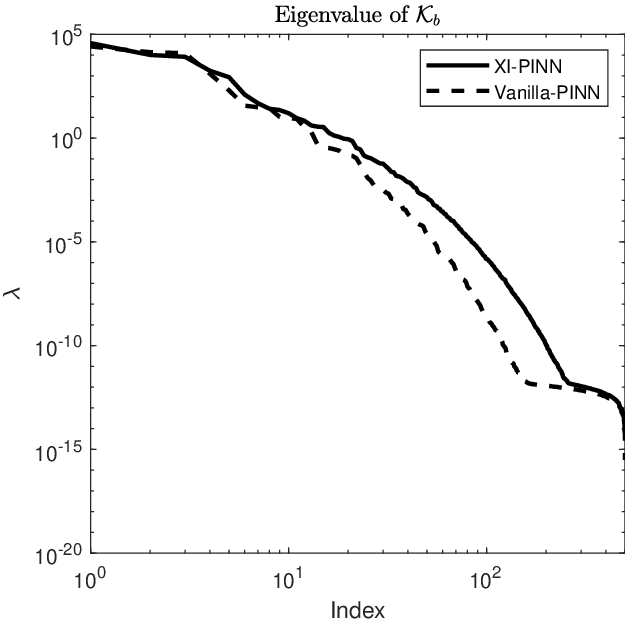}
	\end{minipage}
	\\
	\begin{minipage}{0.25\linewidth}
		\centering
		\includegraphics[width=1\linewidth]{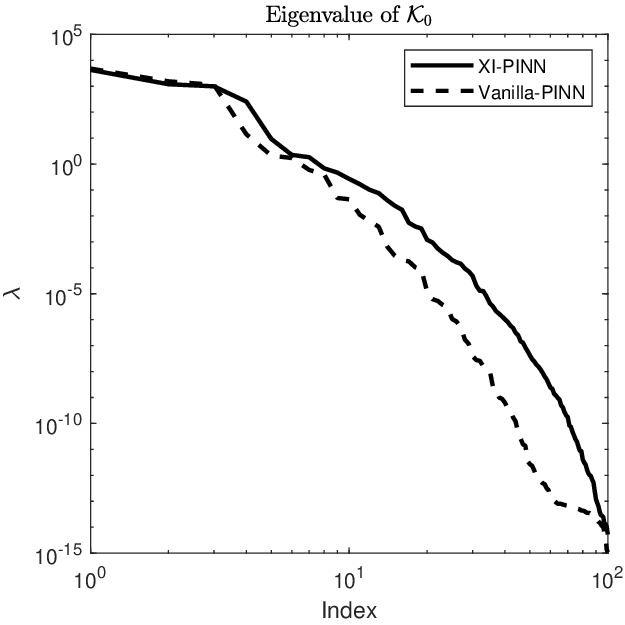}
	\end{minipage}
	\begin{minipage}{0.25\linewidth}
		\centering
		\includegraphics[width=1\linewidth]{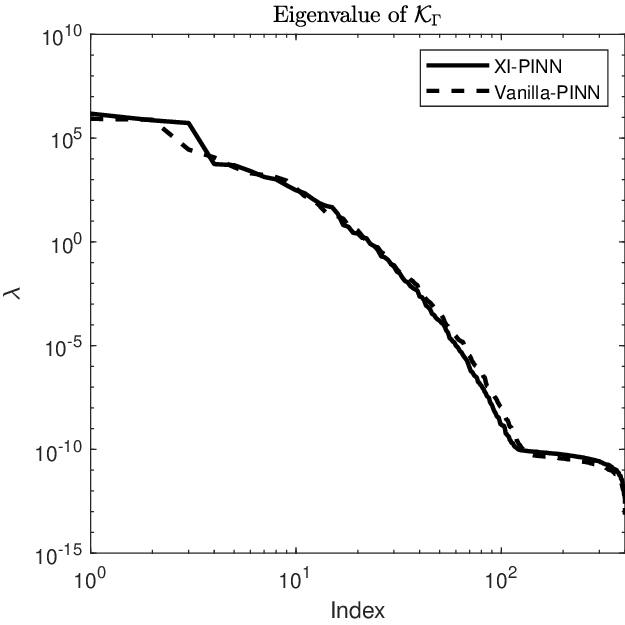}
	\end{minipage}
	\caption{Example~\ref{exa1}: The eigenvalue distribution of the NTK for XI-PINN and Vanilla-PINN.}
	\label{fig4}
\end{figure}

\begin{example} \label{exa2}
	In this example, we consider the moving interface problem (\ref{eq1.1})-(\ref{eq1.7}) defined in 3D space $\Omega = [-1,1]^3$ over the time interval $[0, 1]$. The interface is a rotating and rising ellipsoidal surface. At the initial time $t = 0$, it is represented by the level set function 
	$$\phi_0(x,y,z) = \dfrac{x^2}{0.9^2} + \dfrac{y^2}{0.7^2} + \dfrac{(z + 0.25)^2}{0.5^2} - 1.$$
	
	The motion of the interface can be described by a velocity field $$\mathcal{V} = \left[ -\pi y/2,\pi x /2, 0.5\right] ^T.$$ Therefore, we can obtain the level set function
	\begin{equation*}
		\small
		\begin{split}
			\phi(x,y,z,t) &= \dfrac{\left( x\cos(\pi t /2) + y \sin(\pi t / 2)\right)^2 }{0.9^2} + \dfrac{\left( -x\sin(\pi t /2) + y \cos(\pi t / 2)\right)^2}{0.7^2} \\
			&\qquad\qquad+ \dfrac{\left( z - 0.5t + 0.25\right)^2}{0.5^2} - 1.		
		\end{split}
	\end{equation*}
	The exact solution is as follows:
	\begin{equation*}
		\small
		u(x,y,z,t) = \left\lbrace 
		\begin{aligned}
			&\exp\left[ -\left( x^2 + y^2 + z^2\right) \right] \cos t, \quad &\mathrm{in}\; \Omega^+(t), \\
			&\dfrac{1}{2} \left[ \exp(-t) \exp(z) \sin(\pi x) \sin(\pi y) + 1\right] , \quad  &\mathrm{in}\; \Omega^-(t),
		\end{aligned}
		\right. 
	\end{equation*}
	where $\beta^+ = 10$ and $\beta^- = 1$.
\end{example}

\begin{figure}[htbp]
	\centering
	\label{fig5}
	\includegraphics[width=0.9\linewidth]{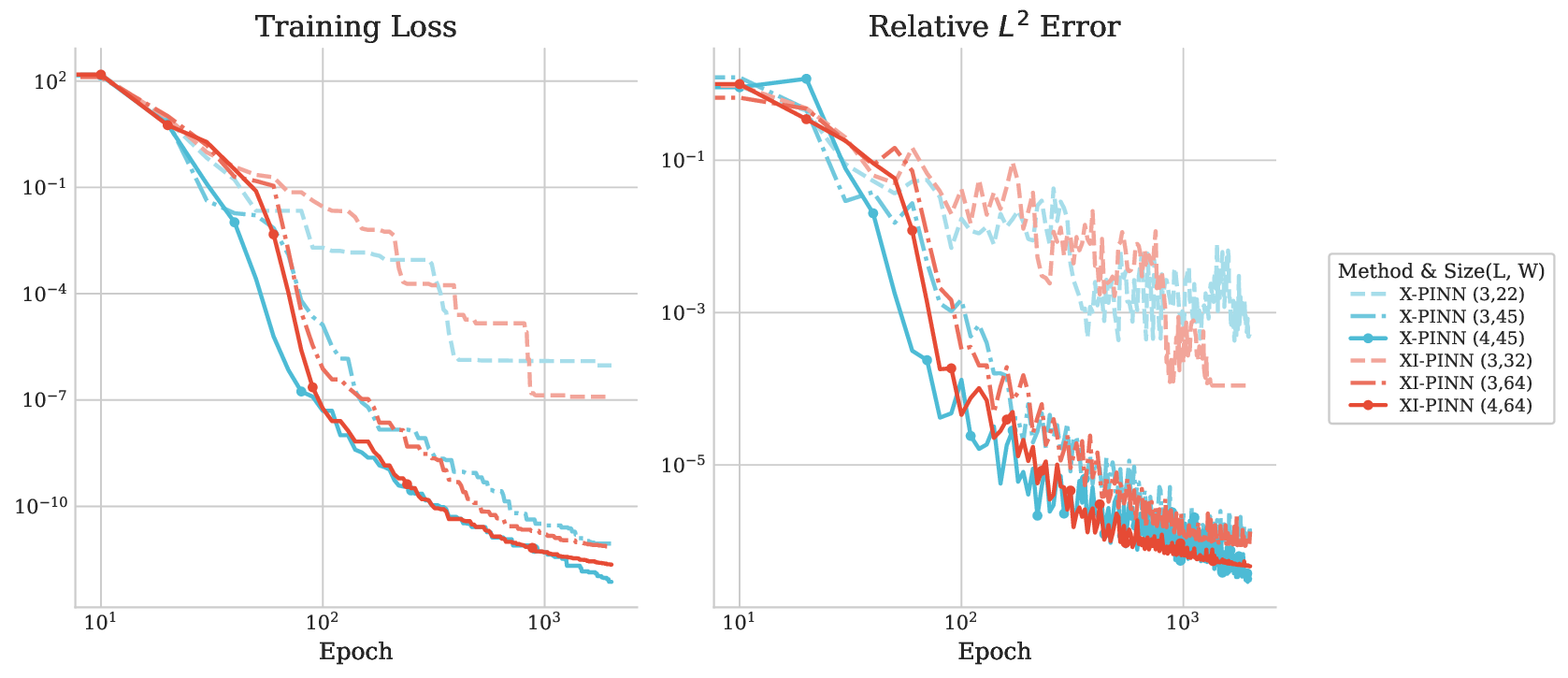}
	\caption{Example~\ref{exa2}: Decay of the loss and the relative $L^2(0,T;L^2(\Omega))$ error during training for different models with various network architectures.}
\end{figure}
\begin{table}[htbp]
	\centering     
	\small        
	\setlength{\tabcolsep}{4pt}
	\caption{Performance of different numerical methods in Example~\ref{exa2}.}
	\label{tab5}
	\begin{tabularx}{\textwidth}{c|*{7}{>{\centering\arraybackslash}X}} 
		\toprule
		Method & $L$ & $W$ & $|\theta|_{\ell^0}$ & $\ell$ & $E^{\text{Rel}}_{L^2}$ & $E^{\text{Rel}}_{H^1}$ & Time (s) \\
		\midrule
		\multirow{3}{*}{X-PINN}
		& 3  & 22 & 1278  & 9.51e-07 & 5.55e-04 & 1.29e-03 & 102.4  \\
		& 3  & 45 & 4592  & 8.91e-12 & 1.16e-06 & 4.96e-06 & 342.7  \\
		& 4  & 45 & 8822  & 7.40e-13 & 3.18e-07 & 1.27e-06 & 1010.0 \\
		\midrule
		\multirow{3}{*}{XI-PINN}
		& 3  & 32 & 1281  & 1.23e-07 & 1.13e-04 & 4.74e-04 & 92.7   \\
		& 3  & 64 & 4545  & 7.01e-12 & 1.32e-06 & 4.49e-06 & 352.3  \\
		& 4  & 64 & 8769  & 2.27e-12 & 4.67e-07 & 2.40e-06 & 1071.9 \\
		\bottomrule
	\end{tabularx}
\end{table}

For problems where the solution is discontinuous across the interface, i.e., $h_D \neq 0$, we choose the indicator function as the extension variable.
In this example we compare XI-PINN with X-PINNs under three different network sizes; the numerical results are shown in Fig.~\ref{fig5} and Table~\ref{tab5}. 
{The numbers of training points are set to $(N_\Omega, N_{\partial\Omega}, N_{\Omega_0}, N_\Gamma) = (20\text{K}, 6\text{K}, 0.6\text{K}, 4\text{K})$.}
{For small network architectures, XI-PINN achieves higher accuracy by sharing parameters across the two subdomains.} Unlike Example~\ref{exa1}, for large networks XI-PINN does not gain an accuracy advantage over X-PINNs. This is because, when each subnetwork of X-PINNs has sufficiently many parameters, it can represent the solution in its own subdomain with high accuracy. In contrast, a pre-assigned small network may suffer from insufficient capacity in one subdomain, which in turn degrades the approximation in the other. Furthermore, when $h_D \neq 0$, XI-PINN no longer benefits from exactly enforcing the zero-jump condition. In such scenarios, both XI-PINN and X-PINNs must learn the interface jump conditions, although XI-PINN uniquely encodes this discontinuity via the extended variable. Consequently, for high-accuracy computations utilizing large networks in the presence of non-homogeneous interface jumps, X-PINNs emerge as a more favorable alternative. The numerical solutions and absolute error distributions on the hyperplane $z=0$ at $t=0.5$ for both methods are given in Fig.~\ref{fig6}. These results also highlight the advantage of neural networks for high-dimensional PDEs; handling such moving interface problems via grid-based space-time finite element methods is a nontrivial task.

\begin{figure}[htbp]
	\centering
	\begin{minipage}{0.49\linewidth}
		\centering
		\includegraphics[width=1\linewidth]{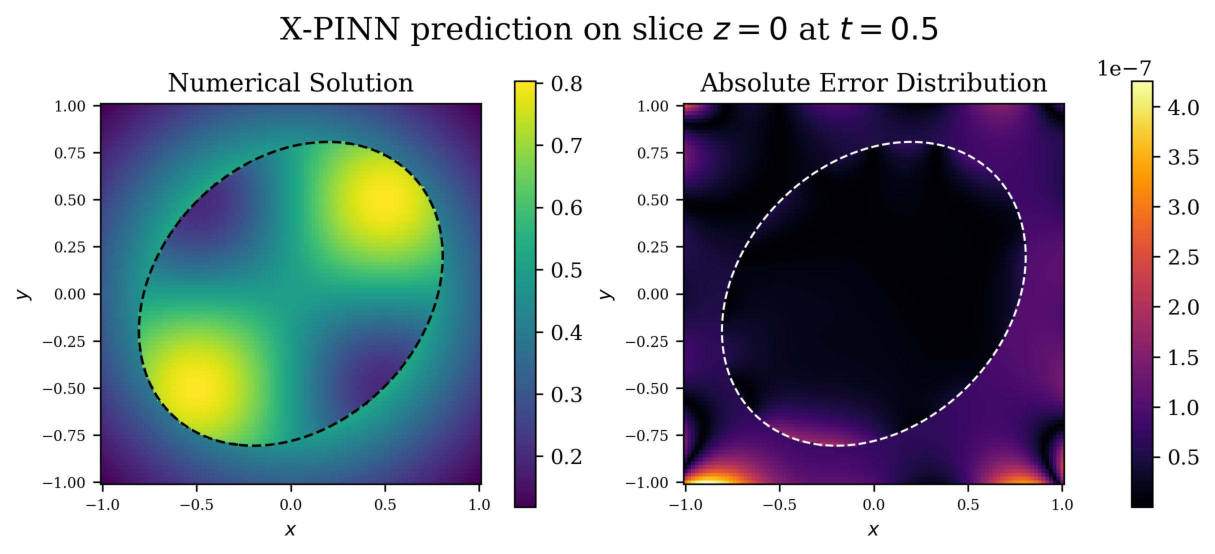}
	\end{minipage}
	\begin{minipage}{0.49\linewidth}
		\centering
		\includegraphics[width=1\linewidth]{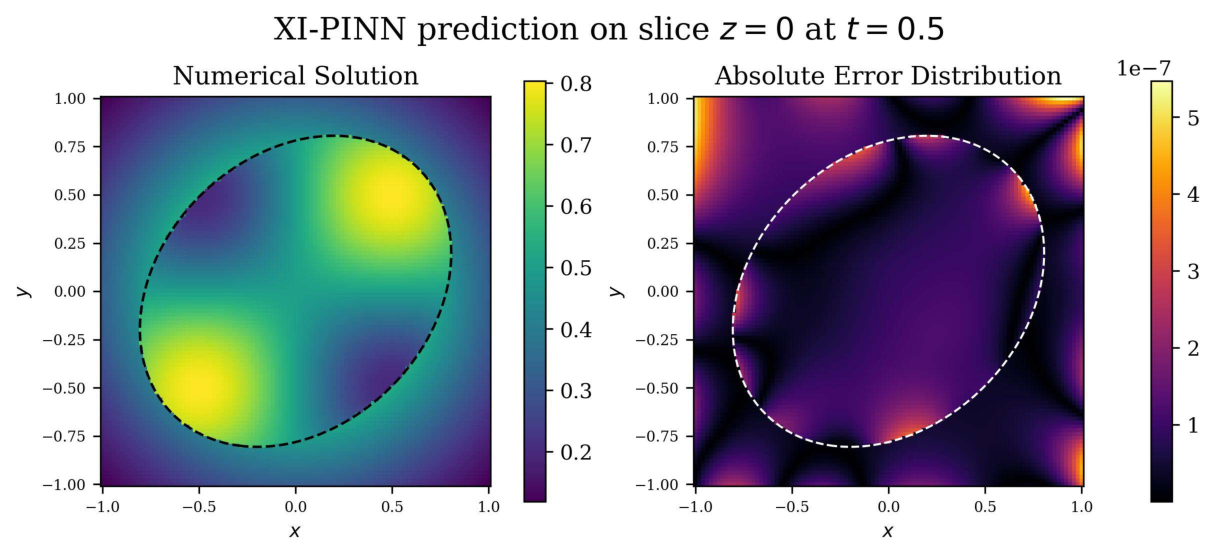}
	\end{minipage}
	\caption{Example~\ref{exa2}: Numerical solutions and absolute error distributions for X-PINN and XI-PINN at $t=0.5$. The network sizes $(L,W)$ are $(4,45)$ and $(4,64)$, respectively.}
	\label{fig6}
\end{figure}

\subsection{Oseen equations}
The XI-PINN method can also be applied to solve more complex time-varying interface problems such as Oseen equations~\cite{ma2023high}. The model is described as follows:
\begin{equation}\label{eq4.1}
	\left\{\begin{aligned}
		\partial_t \bm{u}^{\pm} + (\mathcal{V}\cdot\nabla)\bm{u}^{\pm} - \nu^{\pm} \Delta \bm{u}^{\pm} + \nabla p^{\pm} 
		&= \bm{f}^{\pm}, \qquad \operatorname{div}\bm{u}^{\pm} = 0,
		&& \text{in } \Omega^{\pm}(t), \\
		\bm{u}^{\pm}\big|_{t=0} &= \bm{u}_{0}^{\pm},
		&& \text{in } \Omega^{\pm}(0), \\
		[\nu \partial_{\mathbf{n}}\bm{u} - p \mathbf{n}]_{\Gamma(t)} &= \bm{h}_N, \quad
		[\bm{u}]_{\Gamma(t)} = \bm{h}_D,
		&& \text{on } \Gamma(t),\\
		\bm{u}^{+} &= \bm{g},
		&& \text{on } \partial\Omega,
	\end{aligned}\right.
\end{equation}
where the viscosities $\nu$ is piece-wise positive constants, which is defined as
\begin{equation}\label{eq4.5}
	\nu(\mathbf{x}, t) = 
	\begin{cases}
		\nu^+, & \text{for } \bm{x} \in \Omega^+(t), \\
		\nu^-, & \text{for } \bm{x} \in \Omega^-(t).
	\end{cases}
\end{equation}

\begin{example}\label{exa3}
	Consider the 2D moving interface problem (\ref{eq4.1})--(\ref{eq4.5}) where the domain $\Omega$ is a disk of radius $1$ centered at $(0,0)$.
	The initial interface $\Gamma(0)$ is an irregular smooth curve (see Fig.~\ref{fig8}) and can be represented by the initial level set function
	\begin{equation*}
		\phi_0(x, y) = x^2 + y^2 - 0.4 - \frac{1}{\pi} \sin(\pi x) \cos(\pi y).
	\end{equation*}
	Driven by the velocity field $\mathcal{V} = \left(-y, x\right)^T$, the interface rotates counterclockwise around $(0, 0)$. The final time is set to $T_{\text{end}}=1$.
	Solving equation \eqref{advection equation} by the method of characteristics yields the exact level set function
	\begin{equation}
		\phi(x,y,t) = \xi^2 + \eta^2 - 0.4 - \frac{1}{\pi} \sin(\pi \xi) \cos(\pi \eta),
	\end{equation}
	where
	\begin{equation*}
		\left( \xi, \eta\right) = \left( x\cos(t) + y\sin(t), \, -x\sin(t) + y\cos(t) \right).
	\end{equation*}
	The exact solution $\bm{u} = (u, v)$ is given by
	\begin{equation*}
		\small
		\bm{u}(x, y, t) = 
		\begin{cases}
			\bigl( \exp(x) \sin(y + t),\, \exp(x) \cos(y + t) \bigr)^T + \phi \nabla^{\perp} \phi, & \mathrm{in}\; \Omega^+(t), \\[4pt]
			\bigl( \exp(x) \sin(y + t),\, \exp(x) \cos(y + t) \bigr)^T, & \mathrm{in}\; \Omega^-(t),
		\end{cases}
	\end{equation*}
	where operator $\nabla^{\perp} = (\partial_y, -\partial_x)^T$.  
	The pressure is chosen as
	\begin{equation*}
		p(x, y, t) = 
		\begin{cases}
			\exp(t)\sin(x)\cos(y), & \mathrm{in}\; \Omega^+(t), \\[4pt]
			\cos(t)\sin(x+y),     & \mathrm{in}\; \Omega^-(t),
		\end{cases}
	\end{equation*}
	and the viscosity coefficients are taken as $\nu^+ = 1$ and $\nu^- = 0.1$.
	In this example the interface jump condition satisfies $\bm{h}_D = 0$; that is, the exact velocity $\bm{u}$ is continuous across the interface.
\end{example}

\begin{figure}[htbp]
	\centering
	\includegraphics[width=0.9\linewidth]{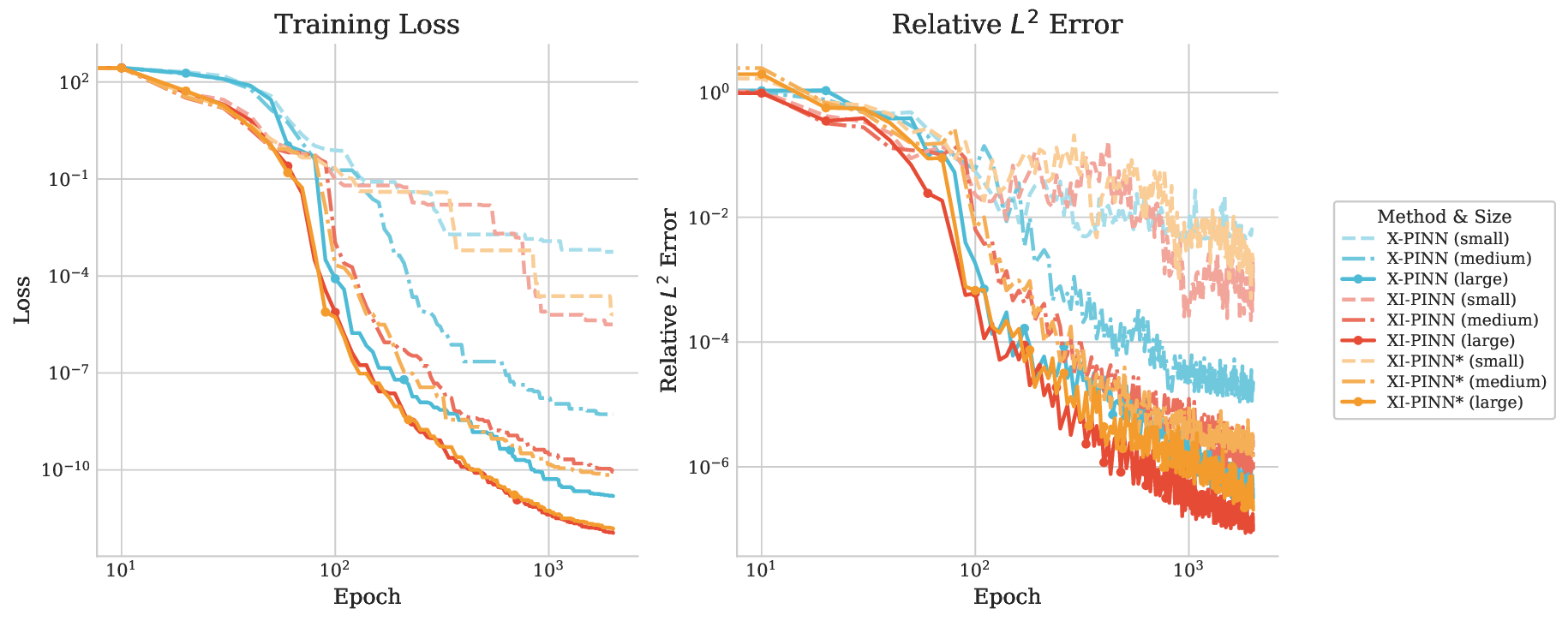}
	\caption{Example~\ref{exa3}: Decay of the loss and the relative $L^2(0,T; L^2(\Omega))$ error of the velocity during training for different models with various network architectures.}
	\label{fig7}
\end{figure}

\begin{table}[htbp]
	\centering
	\smaller                    
	\setlength{\tabcolsep}{4pt}
	\caption{Performance of different numerical methods in Example~\ref{exa3}.}
	\label{tab7}
	\begin{tabularx}{\textwidth}{c|ccc*{5}{>{\centering\arraybackslash}X}} 
		\toprule
		Method & Scale & $|\theta|_{\ell^0}$ & $\ell$ & $E^{\text{Rel}}_{L^2}(u_\theta)$ & $E^{\text{Rel}}_{H^1}(u_\theta)$ & $E^{\text{Rel}}_{L^2}(v_\theta)$ & $E^{\text{Rel}}_{H^1}(v_\theta)$ & Time (s) \\
		\midrule
		\multirow{3}{*}{X-PINN}
		& small  & 2720  & 5.59e-04 & 7.02e-03 & 1.02e-02 & 7.40e-03 & 1.22e-02 & 357.0 \\
		& medium & 6020  & 5.23e-09 & 2.16e-05 & 3.07e-05 & 2.56e-05 & 4.14e-05 & 756.2 \\
		& large  & 11264 & 7.42e-11 & 3.44e-07 & 1.09e-06 & 2.97e-07 & 1.45e-06 & 2065.6 \\
		\midrule
		\multirow{3}{*}{XI-PINN}
		& small  & 2531  & 3.17e-06 & 2.89e-04 & 6.28e-04 & 7.79e-04 & 8.68e-04 & 298.2 \\
		& medium & 5859  & 8.79e-11 & 8.03e-07 & 2.33e-06 & 1.05e-06 & 3.11e-06 & 683.2 \\
		& large  & 11075 & 1.12e-12 & 7.34e-08 & 2.95e-07 & 1.21e-07 & 4.10e-07 & 2023.8 \\
		\midrule
		\multirow{3}{*}{XI-PINN*}
		& small  & 2531  & 6.42e-06 & 1.88e-03 & 8.30e-04 & 4.27e-04 & 7.46e-04 & 358.0 \\
		& medium & 5859  & 6.45e-11 & 8.52e-07 & 2.47e-06 & 8.76e-07 & 3.40e-06 & 741.4 \\
		& large  & 11075 & 1.50e-12 & 1.03e-07 & 3.57e-07 & 1.05e-07 & 4.82e-07 & 2080.4 \\
		\bottomrule
	\end{tabularx}
\end{table}

\begin{table}[htbp]
	\centering
	\footnotesize          
	\renewcommand{\arraystretch}{0.85} 
	\setlength{\tabcolsep}{3.5pt}       
	\caption{Network architectures for Example~\ref{exa3}: $L_{\bm{u}}$, $W_{\bm{u}}$, $L_p$, $W_p$ denote layers and widths for velocity and pressure. The network architecture of $\text{XI-PINN}^*$ is identical to that of XI-PINN.}
	\label{tab8}
	\begin{tabular}{@{}c|c|cccc@{}}
		\toprule
		Method & Scale & $L_{\bm{u}}$ & $W_{\bm{u}}$ & $L_p$ & $W_p$ \\
		\midrule
		\multirow{3}{*}{X-PINN}
		& small  & 3 & 23 & 3 & 23  \\
		& medium & 3 & 45 & 3 & 23  \\
		& large  & 4 & 45 & 4 & 23  \\
		\midrule
		\multirow{3}{*}{XI-PINN}
		& small  & 3 & 32 & 3 & 32  \\
		& medium & 3 & 64 & 3 & 32  \\
		& large  & 4 & 64 & 4 & 32  \\
		\bottomrule
	\end{tabular}
\end{table}

\begin{figure}[htbp]
	\centering
	\begin{minipage}{1 \linewidth}
		\centering
		\includegraphics[width=0.9\linewidth]{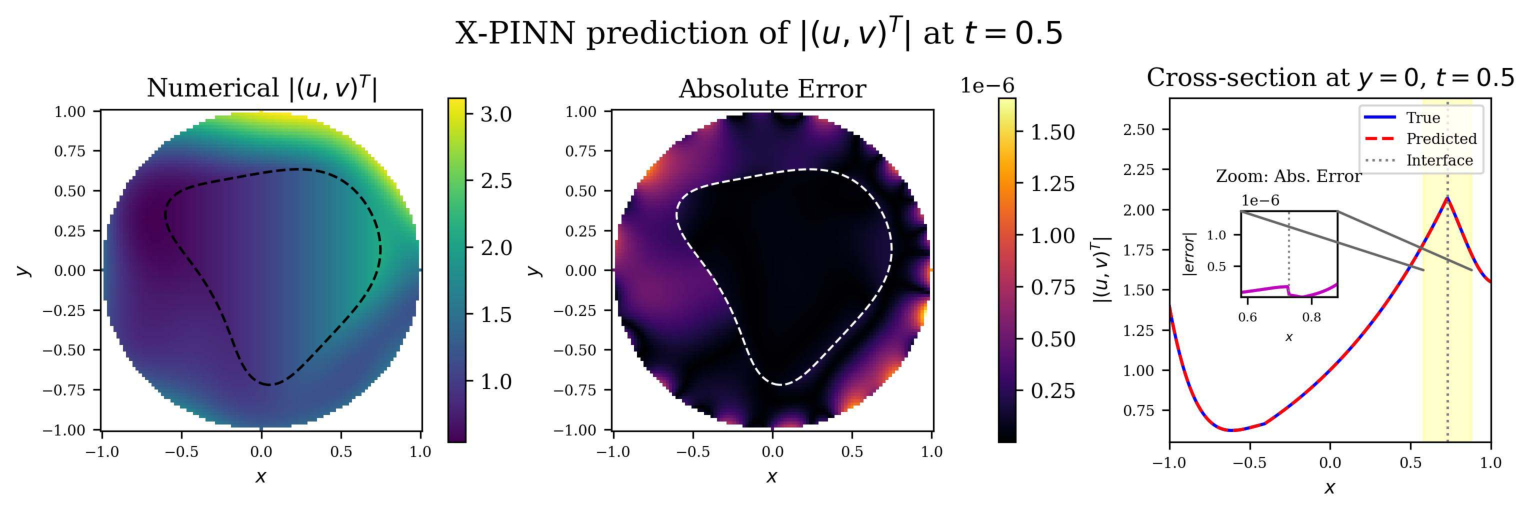}
	\end{minipage}
	\\
	\begin{minipage}{1 \linewidth}
		\centering
		\includegraphics[width=0.9\linewidth]{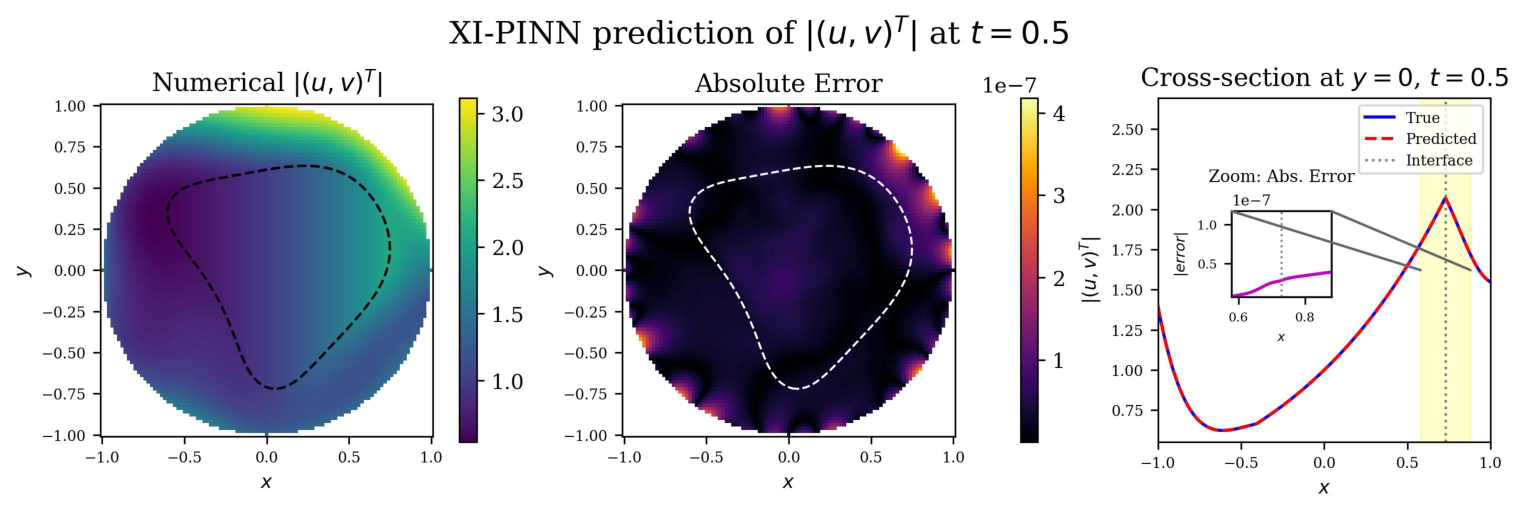}
	\end{minipage}
	\caption{Example~\ref{exa3}: Numerical solution and absolute error distribution for $\left|\bm{u}\right| $ at $t=0.5$.}
	\label{fig8}
\end{figure}

In this example, since the pressure $p$ is discontinuous across the interface, XI-PINN requires two separate neural networks to approximate the velocity $\bm{u}$ and the pressure $p$, respectively.
For the network approximating the velocity, the extension variable is $z = \Phi(\bm{x},t)$, while for the network approximating the pressure we take $z = \chi(\bm{x},t)$.
To ensure a fair comparison, X-PINNs also use different networks for the velocity and for the pressure; hence both methods involve a total of four neural networks. In both XI-PINNs and X-PINNs, the neural network approximating $\bm{u}$ yields two outputs, corresponding to the vector components $(u,v)$.
{In this experiment the numbers of training points are set to $(N_\Omega, N_{\partial\Omega}, N_{\Omega_0}, N_\Gamma) = (10\text{K}, 3\text{K}, 0.3\text{K}, 2\text{K})$.}
{As before, $\text{XI-PINN}^*$ refers to the variant where the level set function is constructed by a neural network $\phi_\theta$; this network has architecture $(L,W)=(2,64)$ and its training takes $59.7$ seconds.}

The decay of the loss and the relative errors during training for the different methods are shown in Fig.~\ref{fig7}, and the final numerical results together with the training times are reported in Table~\ref{tab7}. 
The detailed network architectures are listed in Table~\ref{tab8}.
The numerical behavior is similar to that observed in the scalar case (Example~\ref{exa1}): XI-PINN enforces continuity at the interface by construction and therefore achieves better accuracy than X-PINNs.
From the error distributions at $t=0.5$ (Fig.~\ref{fig8}) we can see that the error of XI-PINN is more uniformly distributed, whereas X-PINNs exhibit noticeable discrepancies between the interior and exterior of the interface.
The cross‑section plots further indicate that such differences may be caused by the error accumulated at the interface.
XI-PINN accurately captures the discontinuity in the normal derivative of $\bm{u}$ across the interface, and its advantage becomes more pronounced when the size of network is small.

\begin{example} \label{exa4}
	In this example, we consider the 2D moving interface problem~(\ref{eq4.1})-(\ref{eq4.5}) involving large deformations of the interface geometry. The domain is given by $\Omega = [0,1 ]^2$. The initial interface $\Gamma(0)$ is the disk whose radius is $0.15$ and its center is $(0.5, 0.75)$ which is determined by the level set function $\phi_0\left( x_1, x_2\right) = x_1^2 + x_2^2 - 0.15^2$. The flow velocity which drives the motion of interface is given by
	$$
	\mathcal{V} = \cos\left( \pi t/3\right) \left( \sin^2\left( \pi x\right)  \sin\left( 2 \pi y\right) , -\sin^2\left( \pi y\right) \sin\left( 2 \pi x\right) \right) ^T.
	$$
	At the finial time $T_{\text{end}}=1$, $\Omega^-(T_{\text{end}})$ is stretched into a snake-shape domain.
	The true solution is set by
	\begin{equation*}
		u(x, y, t) = \left\lbrace 
		\begin{aligned}
			& \left( \exp(x) \sin \left( \pi y + \pi t\right) , \pi^{-1} \exp(x) \cos\left( \pi y + \pi t\right) \right)^T, \quad &\mathrm{in}\; \Omega^+(t), \\
			& \left( \cos\left( \pi x\right) \sin \left( \pi y\right) , -\sin\left( \pi x\right) \cos\left( \pi y\right) \right)^T \cos t, \quad  &\mathrm{in}\; \Omega^-(t),
		\end{aligned}
		\right. 
	\end{equation*}
	\begin{equation*}
		p(x, y, t) = \left\lbrace 
		\begin{aligned}
			& \sin\left( 0.5 \pi x\right) \cos \left( 0.5 \pi y\right) , \quad &\mathrm{in}\; \Omega^+(t), \\
			& \cos\left( 0.5 \pi x\right) \sin \left( 0.5 \pi y\right) , \quad  &\mathrm{in}\; \Omega^-(t),
		\end{aligned}
		\right. 
	\end{equation*}
	and the coefficients of viscosity are taken as $\nu^+=10^{-3}$ and $\nu^- = 1$.
\end{example}

\begin{figure}[htbp]
	\centering
	\includegraphics[width=0.9\linewidth]{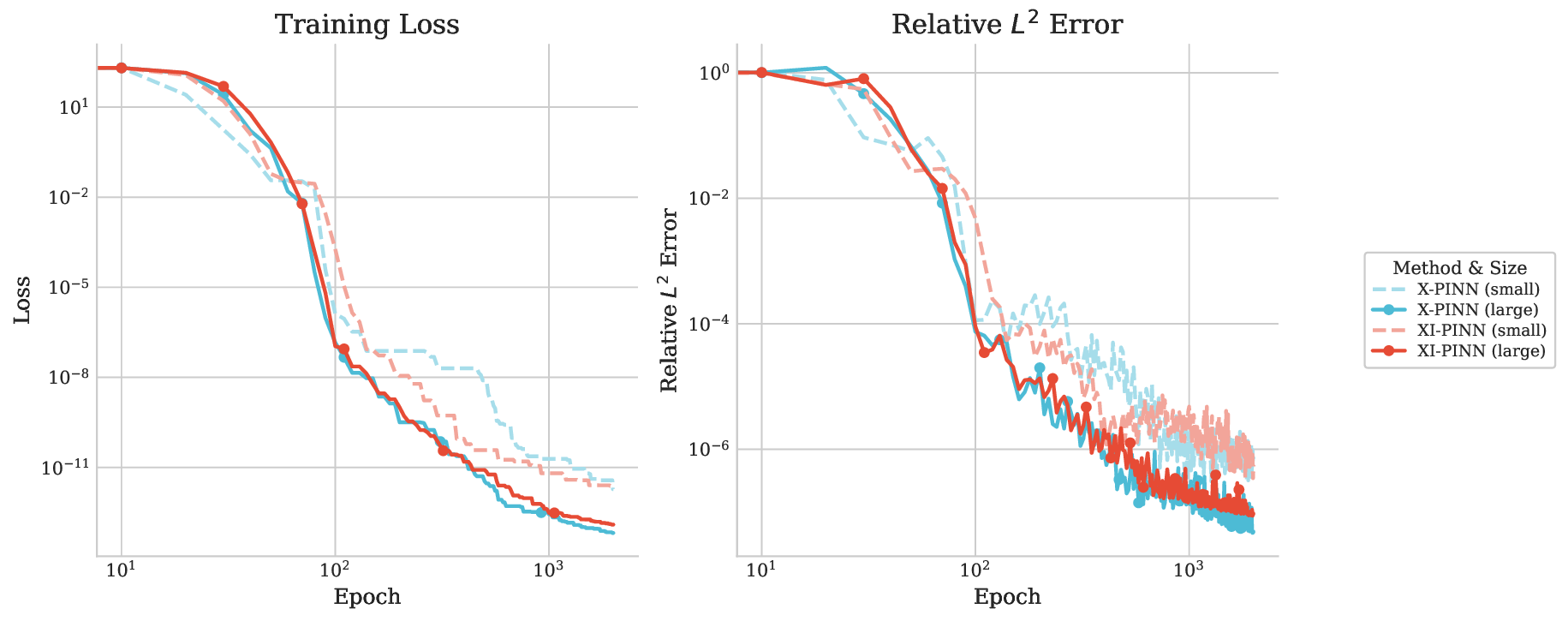}
	\caption{Loss and error curves for Example~\ref{exa4}.}
	\label{fig10}
\end{figure}

\begin{table}[htbp]
	\centering
	\smaller                    
	\setlength{\tabcolsep}{4pt}
	\caption{Performance of different numerical methods in Example~\ref{exa4}.}
	\label{tab9}
	\begin{tabularx}{\textwidth}{c|ccc*{5}{>{\centering\arraybackslash}X}} 
		\toprule
		Method & Scale & $|\theta|_{\ell^0}$ & $\ell$ & $E^{\text{Rel}}_{L^2}(u)$ & $E^{\text{Rel}}_{H^1}(u)$ & $E^{\text{Rel}}_{L^2}(v)$ & $E^{\text{Rel}}_{H^1}(v)$ & Time (s) \\
		\midrule
		\multirow{2}{*}{X-PINN}
		& small & 1326  & 1.90e-12 & 6.41e-07 & 2.59e-07 & 1.42e-06 & 5.85e-07 & 279.8 \\
		& large & 4776  & 6.62e-14 & 3.90e-08 & 1.43e-07 & 9.90e-08 & 4.10e-07 & 571.0 \\
		\midrule
		\multirow{2}{*}{XI-PINN}
		& small & 1315  & 2.45e-12 & 5.44e-07 & 7.09e-07 & 6.95e-07 & 1.72e-06 & 271.5 \\
		& large & 4675  & 1.25e-13 & 7.83e-08 & 2.53e-07 & 1.79e-07 & 6.62e-07 & 578.8 \\
		\bottomrule
	\end{tabularx}
\end{table}
We use this example to illustrate the advantages of neural networks for moving interface problems with large deformations.  
Since both the velocity $\bm{u}$ and the pressure $p$ are discontinuous across the interface, for XI-PINN a single network suffices, with outputs $(u,v,p)$.  
For X-PINNs, two networks are needed, one for each subdomain, to approximate the velocity and the pressure.  
{As an analytical level set function is not available in this example, we construct $\phi_\theta$ using Algorithm~\ref{algorithm 1}.  
Four fully connected networks with architecture $(L,W)=(4,32)$ are employed to build the composite inverse mapping $\widehat{\mathbf{X}}_\theta$; training these four networks takes $366.4$ seconds in total.} For the X-PINN framework, we similarly utilize the learned level set function $\phi_\theta$ to classify the collocation points within the global domain $\Omega$.
{The numbers of training points used for the PDE solution are the same as in Example~\ref{exa3}.}

The numerical results are presented in Fig.~\ref{fig10} and Table~\ref{tab9}.  
{Here XI-PINN is tested with two network sizes: $(L,W)=(3,32)$ and $(3,64)$; for X-PINNs we use $(3,23)$ and $(3,45)$.}
With networks of comparable size, both models achieve high accuracy and the numerical results show no significant difference.  
Because the exact solution $(\bm{u},p)$ is fairly smooth, both methods reach relative errors on the order of $10^{-7}$ even with small networks containing just around $1300$ parameters.  
The numerical solutions and error distributions of $\left| \bm{u}\right| $ at different times are shown in Fig.~\ref{fig11}.  
These results demonstrate that PINN-based methods can handle moving interface problems with large deformations in a simple and effective manner, fully exploiting the advantages of mesh-free approaches.

\begin{figure}[htbp]
	\centering
	\begin{minipage}{1 \linewidth}
		\centering
		\includegraphics[width=1\linewidth]{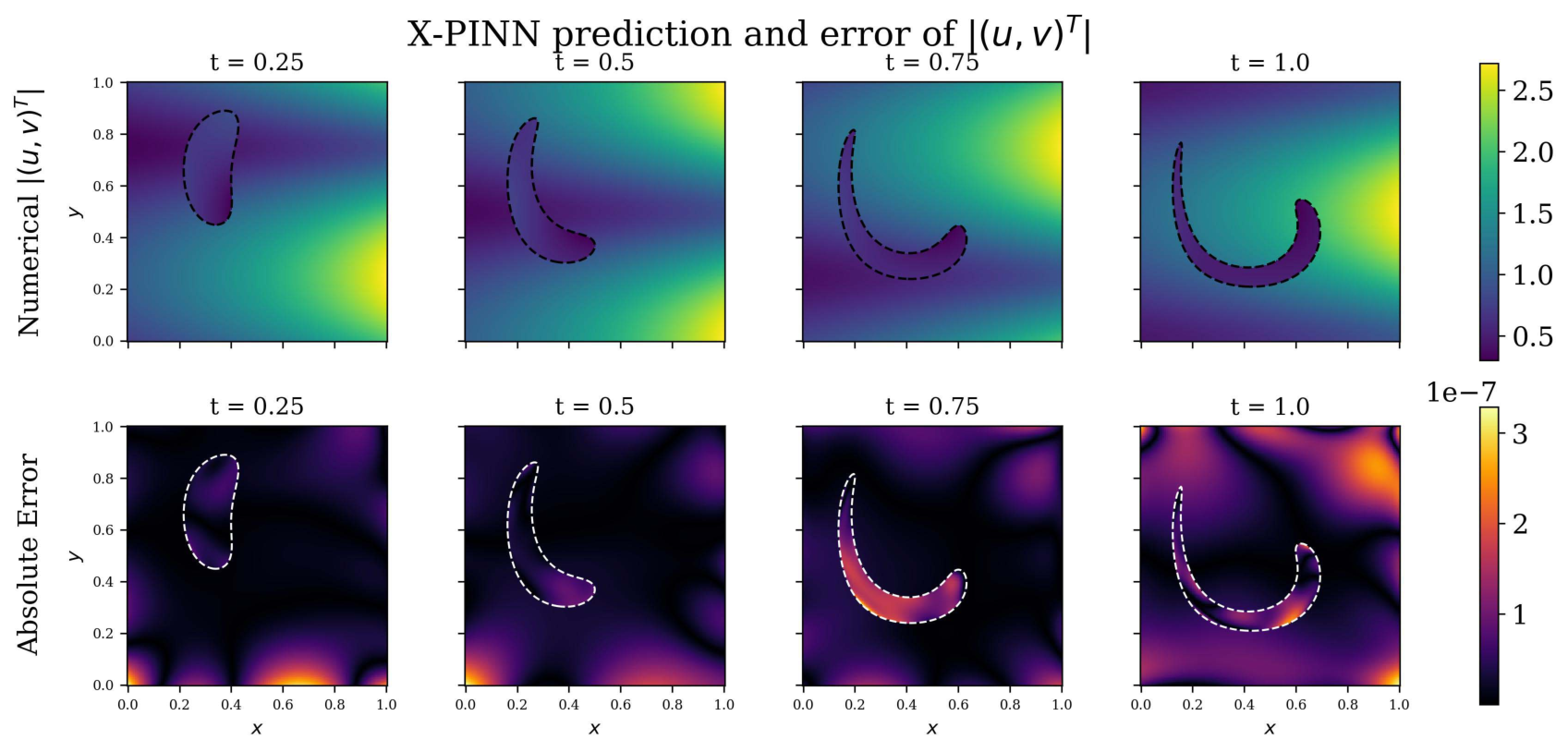}
	\end{minipage}
	\\
	\begin{minipage}{1 \linewidth}
		\centering
		\includegraphics[width=1\linewidth]{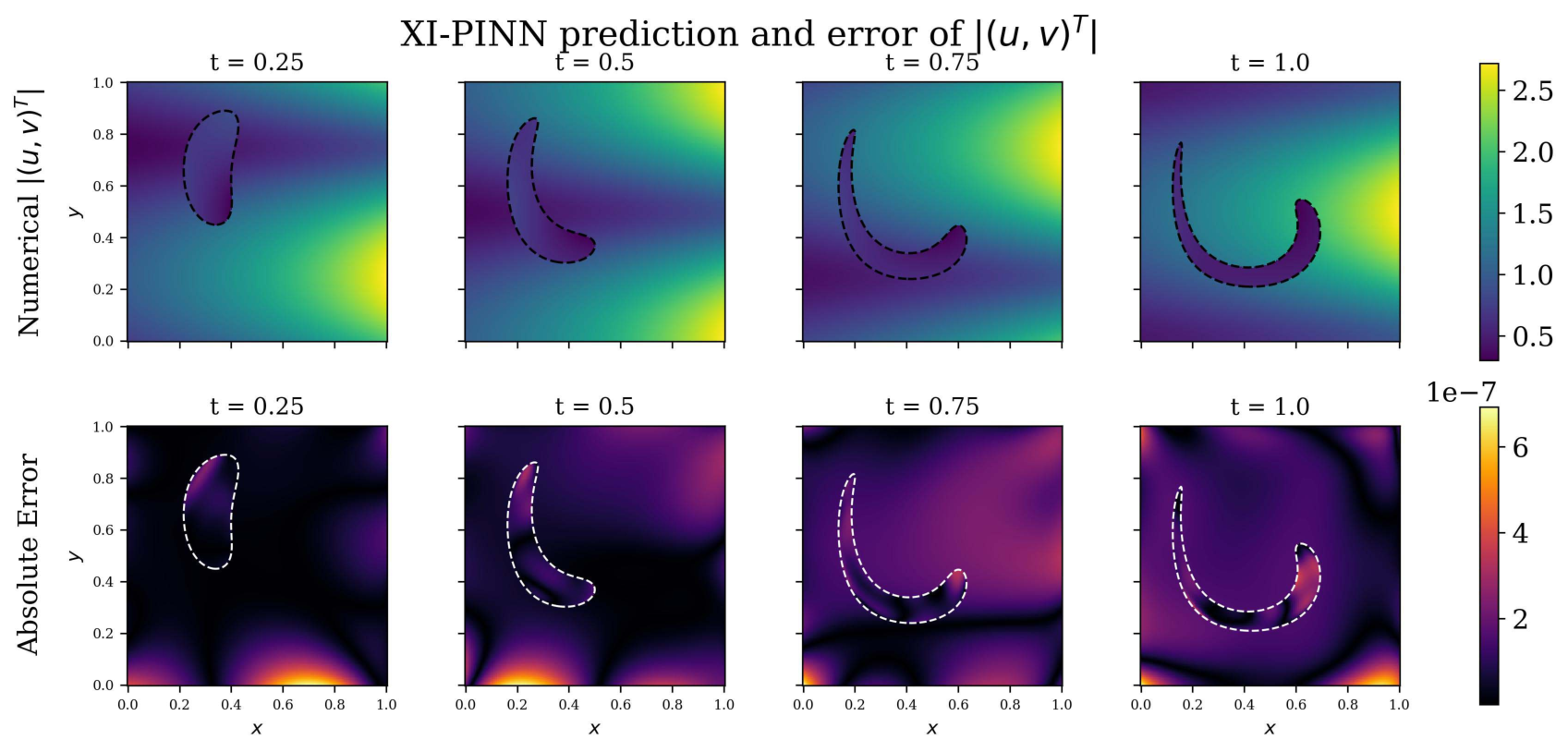}
	\end{minipage}
	\caption{Example~\ref{exa4}: Numerical solution and error distribution for $\left|\bm{u} \right| $ at different times.}
	\label{fig11}
\end{figure}

\section{Conclusion}\label{sec5}
{
In this paper, we have presented the XI‑PINN framework for solving parabolic moving interface problems. The method employs a single neural network combined with an extended variable technique driven by a level set function, allowing it to capture the evolving interface and the associated jump conditions without requiring separate networks for each subdomain or body‑fitted mesh generation. By analyzing the error structure, we have decomposed the total error into approximation, statistical, and optimization components, and provided an a priori error estimate that confirms the method’s reliability under appropriate regularity assumptions.
﻿

Extensive numerical experiments, including two‑ and three‑dimensional examples and the Oseen equations, demonstrate the accuracy and robustness of XI‑PINN. The method consistently outperforms both the standard Vanilla‑PINN, which fails to resolve interface discontinuities, and the multi‑network X‑PINN when the solution is continuous across the interface or when network capacities are limited. The single‑network design enables adaptive parameter sharing, leading to a more uniform error distribution and a favorable spectral behavior of the neural tangent kernel, which alleviates the spectral bias commonly observed in physics‑informed neural networks. Moreover, the time‑adaptive neural‑network construction of the level set function successfully handles large interface deformations while preserving topological validity.
﻿

While the present work focuses on moving interface problems with prescribed velocity fields, promising directions for future research include the extension to fully coupled multiphysics problems, such as two‑phase incompressible Navier–Stokes flows \cite{hysing2009quantitative}, where the interface motion is intrinsically linked to the solution fields. In such settings, XI‑PINN’s ability to naturally enforce velocity continuity across the interface will be particularly advantageous, but it also calls for the development of stable and efficient iterative algorithms that alternately update the PDE solution and the evolving interface representation.
}

\bibliographystyle{siamplain}
\bibliography{references}
\end{document}